\documentclass[12pt,a4paper]{amsart}
\usepackage{mathrsfs}
\usepackage{amsfonts}
\usepackage{txfonts}
\usepackage{hyperref}
\usepackage{latexsym}
\usepackage{amssymb}
\usepackage{upgreek}
\usepackage{hyperref}

\usepackage{epsf}
\newtheorem{theorem}{Theorem}[section]
\newtheorem{lemma}[theorem]{Lemma}
\newtheorem{corollary}[theorem]{Corollary}
\newtheorem{prop}[theorem]{Proposition}

\theoremstyle{definition}
\newtheorem{definition}[theorem]{Definition}

\newtheorem{remark}[theorem]{Remark}

\numberwithin{equation}{section}
\allowdisplaybreaks[4]

\begin{document}

\title[]
{\Large Szlenk and $w^\ast$-dentability indices of C*-algebras}
\author{Lixin Cheng$^\dag$, Zhizheng Yu$^\sharp$ }
\address{$^\dag,  ^\sharp $:  School of Mathematical Sciences, Xiamen University, Xiamen, Fujian 361005, China}
\email{lxcheng@xmu.edu.cn\;\;(L. Cheng)}
\email{19020200156565@stu.xmu.edu.cn\;\;(Z. Yu)}

\thanks{$^\dag$ Support in partial
by NSFC, grant no. 11731010.}
\date{}
\maketitle

{\bf Abstract:}{\small 
~Let $\mathcal A$ be a infinite dimensional C*-algebra and $1<p<\infty$. We compute the Szlenk index of $\mathcal A$ and $L_p(\mathcal A)$, and show that $Sz(\mathcal A)=\Gamma'(i(\mathcal A))$ and $Dz(\mathcal A)=Sz(L_p(\mathcal A))=\omega Sz(\mathcal A)=\omega\Gamma'(i(\mathcal A))$, where $i(\mathcal A)$ is the noncommutative Cantor-Bendixson index, $\Gamma'(\xi)$ is the minimum ordinal number which is greater than $\xi$ of the form $\omega^\zeta$ for some $\zeta$ and we agree that $\Gamma'(\infty)=\infty$ and $\omega\cdot\infty=\infty$. As an application, we compute the Szlenk index [respectively, $w^\ast$-dentability index] of a C*-tensor product $\mathcal A\otimes_\beta\mathcal B$ of non-zero C*-algebras $\mathcal A$ and $\mathcal B$ in terms of $Sz(\mathcal A)$ and $Sz(\mathcal B)$ [respectively, $Dz(\mathcal A)$ and $Dz(\mathcal B)$]. When $\mathcal A$ is a separable C*-algebra, we show that there exists $a\in \mathcal A_h$ such that $Sz(\mathcal A)=Sz(C^\ast(a))$ and $Dz(\mathcal A)=Dz(C^\ast(a))$, where $C^\ast(a)$ is the C*-subalgebra of $\mathcal A$ generated by $a$.
\par


\section[Introduction]{\large Introduction}
In this paper we deal with the Szlenk index of C*-algebra. Let us first recall the definition of the $Szlenk~index$ and the $w^\ast$-$dentability~index$. Suppose $X$ is a Banach space and $K$ is weak*-compact subset of $X^\ast$.
For $\varepsilon>0$, we let
\[
s_\varepsilon(K)=\{x^\ast\in K:{\rm for~any~}w^\ast-{\rm neighborhood}~V~{\rm of}~x^\ast, {\rm diam}(V\cap K)>\varepsilon\}.
\]
We define the  transfinite derived sets by $s_\varepsilon^0(K)=K$, for $\alpha\geq1$
\[
s_\varepsilon^\alpha(K)=\left\{
\begin{array}{rcl}
s_\varepsilon(s_\varepsilon^\beta(K))~~~ &  & {if~\alpha=\beta+1,}\\
\bigcap_{\beta<\alpha}s_\varepsilon^\beta(K)&  & {otherwise.}\\
\end{array}
\right.
\]
We define $Sz_\varepsilon(K)=\min\{\alpha:s_\varepsilon^\alpha(K)=\emptyset\}$ if this class of ordinals is non-empty, and we write $Sz_\varepsilon(K)=\infty$ otherwise. We agree to the convention that $\xi<\infty$ for any ordinal $\xi$ (It's easy to see $Sz_\varepsilon(K)$ is never a limit ordinal when $K\neq\emptyset$.) We let
 \[
 Sz(K)=\sup_{\varepsilon>0}Sz_\varepsilon(K),
 \]
 with the convention that the supremum is $\infty$ if $Sz_\varepsilon(K)=\infty$ for some $\varepsilon$. Last the $Szlenk~index~of~X$ is defined by 
  \begin{equation}\label{1.1}
  Sz(X)=Sz(B_{X^\ast}).
  \end{equation}
When $X$ is a complex Banach space, we let $X_{\mathbb R}$ be the real restriction of $X$. It's easy to see that $Sz(X_{\mathbb R})=Sz(X)$. Now, we recall that a w*-open slice of $X^\ast$ is a subset of $X^\ast$ of the form $\{y\in X^\ast:{\rm Re}~y^\ast(x)>a\}$ for some $a\in\mathbb R$ and $x\in X$. we let
\[
d_\varepsilon(K)=\{x^\ast\in K:{\rm for~any~}w^\ast{\rm -open~slice}~V~{\rm contains}~x^\ast, {\rm diam}(V\cap K)>\varepsilon\}.
\]
We define $d_\varepsilon^\alpha(K)$, $Dz_\varepsilon(K)$, $Dz(K)$ and $Dz(X)$ similarly.

The Szlenk index was first introduced by W. Szlenk \cite{Szlenk} in separable Banach space, in order to show that there is no separable reflexive Banach space universal for the class of all separable reflexive Banach space. The definition of Szlenk index in \cite{Szlenk} is slightly different than (\ref{1.1}), but they are the same for $X$ containing no isomorph of $\ell_1$, can see \cite[Prop 3.3]{Lancien1993}. The Szlenk and $w^\ast$-dentability indices have become an important tool in Banach space theory. A survey \cite{Lancien2006} by G. Lancien for the Szlenk and $w^\ast$-dentability indices and their applications. A remarkable conclusion of the Szlenk index is that it perfectly determines the isomorphism classes of $C(K)$ with $K$ is an infinite compact metrizable space, which can obtain by F. Bessage and A. Pe{\l}zy\'nski \cite{BP}, A.A. Miljutin \cite{Miljutin}, and C. Samuel \cite{Samuel}. On the other hand, A important conclusion of the $w^\ast$-dentability index is that $Dz(X)=\omega$ if and only if $X$ is superreflexive for the nonzero Banach space $X$ \cite{Lancien1995,HMVZ}. The conclusion of C. Samuel is that $Sz(C([0,\omega^{\omega^\alpha}]))=\omega^{\alpha+1}$ whenever $\alpha$ is a countable ordinal. The first extension of C. Samuel's result is by G. Lancien \cite{Lancien1996}, he proved that $Sz(C(K))=\omega^{\alpha+1}$ when $K$ is a compact Hausdorff space satisfying $K^{(\omega^\alpha)}\neq\emptyset$ and $K^{(\omega^{\alpha+1})}=\emptyset$ with $\alpha<\omega_1$. A beautiful proof of C. Samuel's result was given by P. H\'ajek and G. Lancien in \cite{HL2007}, and they computed the Szlenk index of $C([0,\alpha])$ for the ordinal $\omega_1\leq\alpha<\omega_1\omega$, is that $\omega_1\omega$, in \cite{HL2007}. In \cite{HLP2009}, P. H\'ajek, G.Lancien and A. Proch\'azka first compute the $w^\ast$-dentability index of the spaces $C(K)$ where K is a countable compact metric space, namely $Dz(C([0,\omega^{\omega^\alpha}]))=\omega^{1+\alpha+1}$, whenever $0\leq\alpha<\omega_1$. More generally, $Dz(C(K))=\omega^{1+\alpha+1}$ if $K^{(\omega^\alpha)}\neq\emptyset$ and $K^{(\omega^{\alpha+1})}=\emptyset$ with $\alpha<\omega_1$. P.A.H Brooker \cite{B2013} continues to extend the results of C. Samuel, P. H\'ajek, G. Lancien and A. Proch\'azka, showing that $Sz(C([0,\alpha]))=\omega^{\gamma+1}$ and $Dz(C([0,\alpha]))=\omega^{1+\gamma+1}$, where $\gamma$ satisfying $\omega^{\omega^\gamma}\leq\alpha<\omega^{\omega^{\gamma+1}}$. The computation of Szlenk and $w^\ast$-dentability indices for general $C(K)$ spaces with $K$ is a compact Hausdorff space is solved by R.M. Causey in \cite{Causey2017,Causey2017b}. He proved that $Sz(C(K))=\Gamma(i(K))$ and $Dz(C(K))=Sz(L_p(C(K)))=\omega Sz(C(K))=\omega\Gamma(i(K))$, see Theorem \ref{3.1}. Especially to deserve to be mentioned, in \cite{Causey2017}, R.M. Causey proved a beautiful conclusion: If $L$ is a weak* compact set in dual space of some Banach space, then $Sz(\overline{co}^\ast(L))=\Gamma(Sz(L))$. The previous conclusion of the Szlenk index of $C(K)$ is a corollary of this result.

The main purpose of this paper is to compute the Szlenk and $w^\ast$-dentability indices of C*-algebras. In particular, it is a extend the previous result of $C(K)$ spaces with $K$ is compact Hausdorff topology space to the noncommutative C*-algebra. For arbitrary infinite dimension C*-algebra $\mathcal A$, Ghasemi and Koszmider \cite{GK} introduce the noncommutative Cantor-Bendixson index $i(\mathcal A)$ (If $\mathcal A$ is scattered C*-algebra, then $i(\mathcal A)=i(\widehat{\mathcal A})$, see Corollary \ref{4.7}). We show that
\[
Sz(\mathcal A)=\Gamma'(i(\mathcal A))~and~Dz(\mathcal A)=Sz(L_p(\mathcal A))=\omega Sz(\mathcal A)=\omega\Gamma'(i(\mathcal A))
\]
where, $1<p<\infty$, $\Gamma'(\xi)$ is the minimum ordinal number which is greater than $\xi$ of the form $\omega^\zeta$ for some $\zeta$ and we agree to the convention $\Gamma'(\infty)=\infty$ and $\omega\cdot\infty=\infty$. As a application, we show that the Szlenk index [respectively, $w^\ast$-dentability index] of a C*-tensor product $\mathcal A\otimes_\beta\mathcal B$ of non-zero C*-algebras $\mathcal A$ and $\mathcal B$ is $\max\{Sz(\mathcal A),Sz(\mathcal B)\}$ [respectively, $\max\{Dz(\mathcal A),Dz(\mathcal B)\}$]. Inspired by the proof method of H.X. Lin\cite[Th. 5.4]{Lin}, as another application, we show that for arbitrary separable C*-algebra $\mathcal A$, there exists $a\in \mathcal A_h$ such that $Sz(\mathcal A)=Sz(C^\ast(a))$ and $Dz(\mathcal A)=Dz(C^\ast(a))$, where $C^\ast(a)$ is the C*-subalgebra of $\mathcal A$ generated by $a$. After completing this work, we found R.M. Causey provided a short calculation of the Szlenk index of $C(K)$ spaces using the Grasberg norm, see \cite{Causeyarxiv}.

 \section{\large Preliminaries}
Let $X$ be a Banach space, if $Sz(X)<\infty$, it is easy to see that it equivalent to every bounded non-empty subset of $X^\ast$ has non-empty weak* open set of  arbitrarily small diameter. Recall that a real Banach space $X$ is said to be $Asplund$ if every convex continuous function defined on a convex open subset $U$ of $X$ is Fréchet differentiable on a dense $G_\delta$ subset of $U$. If $X$ is a complex Banach space, we say that it is $Asplund$ if $X_{\mathbb R}$, the real restriction of $X$, is Asplund. In fact, the Szlenk index is relevant to the classical theory of Asplund spaces. The following statement summarizes some equivalent properties (e.g. see the book of R. Deville, G. Godefroy and V. Zizler \cite{DGZ} for a complete reference).
\begin{theorem}\label{2.1}
Let $X$ be a Banach space. The following assertions are equivalent:
\begin{itemize}
\item [(a)] $X$ is an Asplund space.
\item [(b)] $Sz(X)<\infty$.
\item [(c)] Every bounded non-empty subset of $X^\ast$ has non-empty weak* open set of  arbitrarily small diameter.
\item [(d)] $Dz(X)<\infty$
\item [(e)] Every bounded non-empty subset of $X^\ast$ has non-empty weak*-slices of  arbitrarily small diameter.
\item [(f)] Every separable subspace of $X$ has a separable dual.
\item [(g)] $X^\ast$ has the Radon-Nikodym property.
\end{itemize}
\end{theorem}
A topological space $T$ is said to be $scattered$ if each of its nonempty (closed) subsets has a relative isolated point. Let us recall the definition of the Cantor-Bendixson derivative of the topological space. For topological space $T$, and for every ordinal number $\alpha$, we define the $\alpha$-$th~derived~set~T^{(\alpha)}$ by transfinite induction: $T^{(0)}=T$; $T^{(1)}=T'=$ the set of all cluster points of $T$;
\[
T^{(\alpha)}=\left\{
\begin{array}{rcl}
\left(T^{(\beta)}\right)'~~~ &  & {if~\alpha=\beta+1,}\\
\bigcap_{\beta<\alpha}T^{(\beta)}&  & {otherwise.}\\
\end{array}
\right.
\]
We let
\[
i(T)=\min\{\alpha:T^{(\alpha)}=\emptyset\},
\]
if this class of ordinals is non-empty, and we write $i(T)=\infty$ otherwise. It's easy to see that $i(T)<\infty$ if and only if $T$ is scattered.

A $C(K)$ space, where $K$ is a compact Hausdorff space, is Asplund if and only if $K$ is a scattered, which was originally shown by I. Namioka and R.R. Phelps \cite{NP}. R.M. Causey's result offers a direct proof of this conclusion. The following theorem:
\begin{theorem}\label{2.2}
Let $K$ is a compact Hausdorff space. Then the following condition are equivalent:
\begin{itemize}
\item [(a)] $K$ is scattered.
\item [(b)] $i(K)<\infty$.
\item [(c)] $Sz(C(K))<\infty$
\item [(d)] $Dz(C(K))<\infty$
\item [(g)] $C(K)$ is Asplund space.
\end{itemize}
\end{theorem}
A. Pe{\l}czy\'nski and Z. Semadendi \cite{PS} provides several equivalent conditions for a compact Hausdorff space to be scattered. 

We know that for arbitrary compact Hausdorff topological space $K$ is scattered if and only if for any nonempty closed subset $L$ of $K$ has a relative isolated point, which equivalent to $C(L)$ has nonzero minimal projection for any nonempty closed subset $L$ of $K$. Which equivalent to $C(L)\cong C(K)/\mathcal J_L$, where $\mathcal J_L=\{f\in C(K):f|_L=0\}$, has nonzero minimal projection for any nonempty closed subset $L$ of $K$. Which equivalent to for every nonzero quotient algebra of $C(K)$, it has nonzero minimal projection. Based on this idea, we can immediately obtain the non-commutative generalization of a scattered compact Hausdorff space to C*-algebras. By $``C^\ast$-$algebra"$ we shall always mean a non-zero C*-algebra.
\begin{definition}
A projection $p$ in a C*-algebra $\mathcal A$ is called $minimal$ if $p\mathcal Ap=\mathbb Cp$. The set of minimal projections of $\mathcal A$ will be denoted by $At(\mathcal A)$. The C*-subalgebra of $\mathcal A$ generated by the minimal projections  of  $\mathcal A$ will  be denoted $\mathcal I^{At}(\mathcal A)$.
\end{definition}
\begin{definition}\label{scatteredc}
A C*-algebra $\mathcal A$ is called $scattered$ if for every nonzero quotient C*-algebra of $\mathcal A$, it has nonzero minimal projection.
\end{definition}
If $\mathcal A$ is a non-zero C*-algebra, we denote by $\widehat{\mathcal A}$ the set of unitary equivalence classes of non-zero irreducible representation of $\mathcal A$. If $(H,\pi)$ is a non-zero irreducible representation of $\mathcal A$ we denote its unitarily equivalence class in $\widehat{\mathcal A}$ by $[H,\pi]$, and we set $\ker[H,\pi]=\ker(\pi)$.

In 1963, J. Tomiyama studied C*-algebras with separable dual in \cite{Tomiyama}, which it is Asplund. The notion of a scattered C*-algebra was first introduced independently by H. Jensen \cite{Jensen1977} and M.L. Rothwell \cite{Rothwell}, which is different from the definition in Definition \ref{scatteredc} but equivalent to it, see \cite{GK}. In \cite{Chu1981}, C. Chu proved that a C*-algebra is scattered if and only if its dual has the Radon-Nikodym property, that is, it is Asplund. The reader is referred to \cite{Chu1982,Chu1994,Huruya,Jensen1978,Jensen1979,Kusuda2010,Kusuda2012,Lin,Wojtaszczyk}  for other equivalent conditions on scattered C*-algebras and some conclusions about scattered C*-algebras. The following statement summarizes which we will use the conclusions of the scattered C*-algebra.
\begin{prop}\label{2.5}
	Suppose that $\mathcal A$ is a C*-algebra. The following conditions are equivalent:
	\begin{itemize}
	    \item [(a)]
	$\mathcal A$ is scattered.
	    \item [(b)]
	$\mathcal A$ is Asplund space.
            \item [(c)]
        $Sz(\mathcal A)<\infty$.
           \item [(d)]
        $Dz(\mathcal A)<\infty$.
 	    \item [(e)]
	$\mathcal A$ is of type I and $\widehat{\mathcal A}$ is scattered.
	    \item [(f)]
	$\mathcal A$ dose not contains a copy of the C*-subalgebra $C_0((0,1])$.
	   \item [(g)]
	Every C*-subalgebra of $\mathcal A$ is scattered.
	  \item [(h)]
    Every quotient C*-algebra of $\mathcal A$ is scattered.
	\end{itemize}
\end{prop}
The following theorem can be found in H. Jensen \cite{Jensen1978}:
\begin{theorem}[\text{H. Jensen}]\label{2.6}
Let $\mathcal A$ be a separable or type I C*-algebra. Then $\mathcal A$ is scattered if and only if $\widehat{\mathcal A}$ is scattered.
\end{theorem}
The following theorem, which are due to S. Ghasemi and P. Koszmide \cite{GK}.
\begin{theorem}[\text{S. Ghasemi and P. Koszmide}]\label{2.7}
	Let $\mathcal A$ be a C*-algebra. Then
\begin{itemize}
	\item [(a)] 
$\mathcal I^{At}(\mathcal A)$ is an ideal of $\mathcal A$.
    \item [(b)]
$\mathcal I^{At}(\mathcal A)$ is isomorphic to *-subalgebra of the algebra $\mathcal K(H)$ of all compact operators on a Hilbert space $H$, i.e. there exists Hilbert space $\{H_i\}_{i\in I}$ such that $\mathcal I^{At}(\mathcal A)$ is isomorphic to the $c_0$-sum of $\{\mathcal K(H_i)\}_{i\in I}$.
    \item [(c)]
$\mathcal I^{At}(\mathcal A)$ contains all ideals of $\mathcal A$ which are isomorphic to a *-subalgebra of $\mathcal K(H)$ for some Hilbert space $H$.
    \item [(d)]
If $\mathcal J$ is an ideal of $\mathcal A$, then $\mathcal I^{At}(\mathcal J)=\mathcal I^{At}(\mathcal A)\cap\mathcal J$.
\end{itemize}
\end{theorem}
\begin{definition}
	If $S$ is a subset of a C*-algebra $\mathcal A$, we let
	\[
	{\rm hull}'(S)=\big\{[H,\pi]\in\widehat{\mathcal A}:\ker[H,\pi]\supset S\big\}.
	\]
	If $R$ is a non-empty set of $\widehat{\mathcal A}$, we denote
	\[
	\ker'(R)=\bigcap_{[H,\pi]\in R}\ker\pi.
	\]
We set $\ker'(\emptyset)=\mathcal A$.
\end{definition}
We will use some classical theory of the spectrum of C*-algebra in the following text.  A good comprehensive reference is the book of J. Dixmier \cite{Dixmier}. The following theorem, it is see to obtain by basic conclusions which the theory of the spectrum of C*-algebra and the conclusion of J. Dixmier \cite{Dixmier1964}.

\begin{theorem}\label{2.9}
	Let $\mathcal A$ be of separable or type I C*-algebra. Then $\mathcal A$ is isomorphic to the $c_0$-sum of $\{\mathcal K(H_i)\}_{i\in I}$, where $\{H_i\}_{i\in I}$ is a family of Hilbert spaces,  if and only if $\widehat{\mathcal A}$ is discrete.
\end{theorem}

\begin{definition}
	Let $\mathcal A$ be a C*-algebra. A von Neumann algebra $\mathcal M$ is called universal enveloping von Neumann algebra of $\mathcal A$, if the following conditions hold
	\begin{itemize}
		\item [(1)] $\mathcal A\subset\mathcal M$,
		\item [(2)] $\mathcal A$ is $\sigma$-weakly dense in $\mathcal M$,
		\item [(3)] $\{\varphi|_\mathcal A:\varphi\in\mathcal M_\ast\}=\mathcal A^\ast$.
	\end{itemize}
\end{definition}
Such the von Neumann algebra is definitely exist (and uniquely determined up to isomorphism). A well-know is that 
\[
\mathcal M=\left({\bigoplus}_{\varphi\in\mathcal S(\mathcal A)}\left(H_\varphi,\pi_\varphi\right)\right)\Big(\mathcal A\Big)'',
\]
(e.g. see \cite{Takesaki}). In particular, by Kaplansky density theorem, $T:\varphi\in\mathcal M_\ast\mapsto\varphi|_\mathcal A\in\mathcal A^\ast$ is linearly isometric surjective map from $\mathcal M_\ast$ to $\mathcal A^\ast$, then $T^\ast$ is $w^\ast$-$w^\ast$ linearly isometric surjective map from $\mathcal A^{\ast\ast}$ to $\mathcal M$. Hence, without loss of generality, we can see that $\mathcal A^{\ast\ast}$ is a von Neumann algebra.

\section{\large Szlenk and $w^\ast$-dentability indices of Banach spaces}
Next, we list some conclusions concerning the Szlenk index.
\begin{prop}\label{3.1}
Let $X$ be a Banach space.
	\begin{itemize}
	    \item [(a)] $Sz(X)=1$ if and only if $\dim X<\infty$.
		\item [(b)]
If $Y$ is a closed subspace of $X$, then
\[
	Sz(Y), Sz(X/Y)\leq Sz(X).
\]
        \item [(c)]
        If $K\subset X^\ast$ is a non-empty $w^\ast$-compact convex set and $Sz(K)<\infty$, then there exists an ordinal $\xi$ such that 
\[
Sz(K)=\omega^\xi.
\]
In particular, if $Sz(X)<\infty$, then $Sz(X)=\omega^{\zeta}$ for some ordinal $\zeta$. The same is true for the $w^\ast$-dentability index.

        \item [(d)]
If $K,L\subset X^\ast$ are w*-compact and $K\subset L$, then
\[
Sz(K)\leq Sz(L).
\]
        \item [(e)]
For any $A,K,L\subset X^\ast$ nonempty w*-compact set, ordinal $\xi$ and $m\in\mathbb N$, if $A\subset K+L$, then
\[
s_{4\varepsilon}^{\omega^\xi m}(A)\subset\bigcup_{i+j=m}\big[s_\varepsilon^{\omega^\xi i}(K)+s_\varepsilon^{\omega^\xi j}(L)\big]
\]
        \item [(f)]
Let $\alpha$ be an ordinal. Assume that
\[
\forall\varepsilon>0~\exists\delta(\varepsilon)>0~s_\varepsilon^\alpha\big(B_{X^\ast}\big)\subset(1-\delta(\varepsilon))B_{X^\ast}.
\]
Then
\[
Sz(X)\leq\alpha\cdot\omega.
\]
      \end{itemize}
\end{prop}
Item (a) can be found in \cite{B2012}. Item (b): It is easy to see that $Sz(X/Y)=Sz(B_{Y^\perp})\leq Sz(X)$, and $Sz(Y)\leq Sz(X)$ is given in \cite{HLM}. Item (c) is given in \cite{Lancien1996} when $K=B_{X^\ast}$, but the same proof works for convex sets. Another proof can be found in \cite{Causey2017}. Item (d) follow easily from the definition. Item (e) can see \cite[Lem. 2.4]{Causey2017b}. Item (f) can be found in \cite[Prop. 2.2]{HL2007} or \cite[Prop. 3.4]{Lancien2006}.

Let $X$ be a Banach space, and given a $p\in(1,\infty)$, $L_p(X)$ denotes the space of (equivalence classes of) Bochner-integrable, $X$-valued functions defined $[0,1]$ with Lebesgue measure $f$ such that $\int\Vert f\Vert^p<\infty$. If $Y$ is a closed subspace of $X$, a natural inclusion
by given
\[
L_p(Y)\equiv\{f\in L_p(X):f(t)\in Y~{\rm a.e.}~t\in[0,1]\}\subset L_p(X).
\]
We denote $Q_Y:L_p(X)\to L_p(X/Y)$ given by $(Q_Yf)(t)=f(t)+Y$. Obviously $\ker(Q_Y)=L_p(Y)$. By the density of the simple functions in $L_p(X/Y)$ and $\Vert Q_Y\Vert\leq1$, it is easy to see that $\overline{Q_Y(B_{L_p(X)})}=B_{L_p(X/Y)}$. This implies that $Q_Y$ is a quotient map and $L_p(X)/L_p(Y)\cong L_p(X/Y)$.

When $X$ is Asplund, i.e. $X^\ast$ has the Radon-Nikodym property see Theorem \ref{2.1}. It follows from \cite[Page 98]{DU}, for any $1<p<\infty$, $L_p(X)^\ast\cong L_{p'}(X^\ast)$ via the canonical embedding of $L_{p'}(X^\ast)$ into $L_p(X)^\ast$, where $1/p+1/p'=1$. If $Y$ is a closed space of $X$, it is easy to see that $L_p(Y)^\perp\cong L_{p'}(Y^\perp)$ via the canonical embedding. Indeed, one way is that $L_p(Y)^\perp\cong\big(L_p(X)/L_p(Y)\big)^\ast\cong \big(L_p(X/Y)\big)^\ast\cong L_{p'}\big((X/Y)^\ast\big)\cong L_{p'}(Y^\perp)$.

The following theorem was given in \cite{Lancien2006} when $p=2$, but it can be easily adjusted to any $1<p<\infty$, or can see \cite[Thm 5.2]{Causey2017a}:
\begin{theorem}\label{3.2}
Let $X$ be a Banach space and $1<p<\infty$. Then 
\[
Dz(X)\leq Sz(L_p(X)).
\]
\end{theorem}
\begin{remark}
The further relationship between the Szlenk and $w^\ast$-dentability indices can see \cite{Causey2017a}, \cite{Causey2022a} and \cite{HS}.
\end{remark}

\begin{definition}\label{3.4}
Let $X$ be a Banach space and $\{P_i\}_{i\in I}$ is a family of projections on $X$. We said $\{P_i\}_{i\in I}$ is $uniformly~convex$ if $\Vert P_i\Vert\leq1$ for each $i\in I$, and for any $\varepsilon>0$, there exists $\delta(\varepsilon)>0$ such that $\Vert x-P_ix\Vert<\varepsilon$ for every $i\in I$, whenever $\Vert x\Vert=1$ and $\Vert x+P_ix\Vert>2-\delta(\varepsilon)$.
\end{definition}

\begin{theorem}\label{3.5}
Let $X$ be a Banach space and $\{P_i\}_{i\in I}$ is a family of projections on $X$. The following are equivalent:
\begin{itemize}
\item [(a)]
$\{P_i\}_{i\in I}$ is uniformly convex.
\item [(b)]
For any $\varepsilon>0$, there exists $\delta(\varepsilon)>0$ such that $\Vert x-P_ix\Vert<\varepsilon$ for every $i\in I$, whenever $x\in B_X$ and $\Vert x\Vert-\Vert P_ix\Vert<\delta(\varepsilon)$.
\item [(c)]
For any $\varepsilon>0$, there exists $\delta(\varepsilon)>0$ such that $\Vert x-P_ix\Vert<\varepsilon$ for every $i\in I$, whenever $\Vert x\Vert=1$ and $\Vert P_ix\Vert>1-\delta(\varepsilon)$.
\end{itemize}
\end{theorem}
\begin{proof}
(a)$\Rightarrow$(b): For any $\varepsilon>0$, $i\in I$, $x\in X$ with $\Vert x\Vert\leq1$ and $\Vert x\Vert-\Vert P_ix\Vert<\varepsilon\delta(\varepsilon)/4$. If $\Vert x\Vert<\varepsilon/2$, then
\[
\Vert x-P_ix\Vert\leq2\Vert x\Vert<\varepsilon.
\]
If $\Vert x\Vert\geq\varepsilon/2$. We let $y=x/\Vert x\Vert$, then
\[
\Vert y+P_iy\Vert\geq\Vert P_i(y+P_iy)\Vert=2\Vert P_iy\Vert=2-2(\Vert y\Vert-\Vert P_iy\Vert)>2-\delta(\varepsilon)
\]
It follows that $\Vert x-P_ix\Vert\leq\Vert y-P_iy\Vert<\varepsilon$.

(b)$\Rightarrow$(c) is obvious. 

(c)$\Rightarrow$(a): First, we show that $\Vert P_i\Vert\leq1$ for each $i\in I$. If not, then there exists $\Vert x\Vert=1$ such that $\Vert P_ix\Vert>1$, this implies that $\Vert x-P_ix\Vert<\varepsilon$ for all $\varepsilon>0$. It follows that $P_ix=x$, which contradicts that $\Vert P_ix\Vert>1=\Vert x\Vert$. Now, for any $\varepsilon>0$, if $x\in X$ with $\Vert x\Vert=1$ and $\Vert x+P_ix\Vert>2-\delta(\varepsilon)$, then
\[
\Vert P_ix\Vert\geq\Vert x+P_ix\Vert-\Vert x\Vert>1-\delta(\varepsilon).
\]
It follows that $\Vert x-P_ix\Vert<\varepsilon$.
\end{proof}

\begin{definition}\label{3.6}
Let $X$ be a Banach space. $\{Y_i\}_{i\in I}$ is a family of closed subspaces of $X$ is said to satsifying $(\ast)$-$condition$ if there is a uniformly convex family of projections $\{P_i\}_{i\in I}$ on $X^\ast$ such that $\ker(P_i)={Y_i}^\perp$ for each $i\in I$. When $\{Y_i\}_{i\in I}$ is just a single subspace $Y_0$, we can also say that it satisfies $(\ast)$-$condition$ similarly.
\end{definition}

We denote by
\[
\delta_{\{Y_i\}}(\varepsilon)=\inf\big\{\Vert\varphi\Vert-\Vert P_i\varphi\Vert:i\in I,~\varphi\in B_{X^\ast},~\Vert\varphi-P_i\varphi\Vert\geq\varepsilon\big\},~~\forall\varepsilon\in(0,\infty),
\]
(with the convention that $\inf(\emptyset)=\infty$). By Theorem \ref{3.5}, we know that $\delta_{\{Y_i\}}(\varepsilon)>0$. A simple example is that $\{Y_i\}_{i\in I}$ is a family of M-ideals of $X$, i.e. for any $\varphi\in X^\ast$, $\Vert\varphi\Vert=\Vert P_i\varphi\Vert+\Vert\varphi-P_i\varphi\Vert$ for all $i\in I$. Since $\Vert P_i\Vert\leq1$, it is easy to see that $P_i(X^\ast)\cong X^\ast/{Y_i}^\perp\cong {Y_i}^\ast$ by setting
\[
\varphi\in P_i(X^\ast)\longmapsto\varphi|_{Y_i}\in {Y_i}^\ast.
\]

\begin{prop}\label{3.7}
Let $X$ be a Banach space and $\{Y_i\}_{i\in I}$ be a family of subspaces of $X$ satisfying $(\ast)$-condition.
\begin{itemize}
\item [(a)]
If there is a closed subspace of $Y$ of  $X$ such that $Y_i\subset Y$ for all $i\in I$. Then $\{Y_i\}_{i\in I}$ is a family of closed subspaces of $Y$ satisfying $(\ast)$-condition.
\item [(b)]
If there is a closed subspace $Z$ of $X$ such that $Z\subset Y_i$ for all $i\in I$. Then $\{Y_i/Z\}_{i\in I}$ is a family of closed subspaces of $X/Z$ satisfying $(\ast)$-condition.
\item [(c)]
If $X$ is Asplund and given a $p\in(1,\infty)$, then $\{L_p(Y_i)\}_{i\in I}$ is a family of closed subspaces of $L_p(X)$ satisfying $(\ast)$-condition.
\end{itemize}
\end{prop}
\begin{proof}
\begin{itemize}
\item [(a):]
We know that $X^\ast/Y^\perp\cong Y^\ast$ and 
\[
{Y_i}^\perp/Y^\perp\cong\{\phi\in Y^\ast:\phi|_{Y_i}=0\}
\]
by natural isomorphism. We denote
\[
Q_i:\varphi+Y^\perp\in X^\ast/Y^\perp\longmapsto P_i(\varphi)+Y^\perp,
\]
which is a projection on $X^\ast/Y^\perp$ induced by $P_i$ with $\ker(Q_i)={Y_i}^\perp/Y^\perp$ for all $i\in I$. Since $\Vert P_i\Vert\leq1$, this implies that $\Vert P_i(\varphi)+Y^\perp\Vert=\Vert P_i(\varphi)\Vert$. For any $\varepsilon>0$, if $\varphi+Y^\perp\in B_{X^\ast/Y^\perp}$ with $\Vert\varphi+Y^\perp\Vert-\Vert Q_i(\varphi+Y^\perp)\Vert<\delta_{\{Y_i\}}(\varepsilon)$. By the $w^\ast$-compactness of $B_{Y^\perp}$, without loss of generality we can assume that $\Vert\varphi+Y^\perp\Vert=\Vert\varphi\Vert$. Therefore, $\Vert\varphi\Vert-\Vert P_i(\varphi)\Vert<\delta_{\{Y_i\}}(\varepsilon)$, it follows that $\Vert(\varphi+Y^\perp)-Q_i(\varphi+Y^\perp)\Vert=\Vert\varphi-P_i(\varphi)+Y^\perp\Vert\leq\Vert\varphi-P_i\varphi\Vert<\varepsilon$. Therefore, $\{Y_i\}_{i\in I}$ is a family of closed subspaces of $Y$ satisfying $(\ast)$-condition by Theorem \ref{3.5}.
\item[(b):]
We know that $Z^\perp\cong(X/Z)^\ast$, and $(Y_i/Z)^\perp\cong {Y_i}^\perp$ for all $i\in I$ by natural isomorphism. It is easy to see that $\{Y_i/Z\}_{i\in I}$ is a family of closed subspaces of $X/Z$ satisfying $(\ast)$-condition.
\item[(c):] Let $\{P_i\}$ is a family of projection on $X^\ast$ with the property in Definition \ref{3.6}. We denote $Q_i$ is the projection on $L_p(X)^\ast=L_{p'}(X^\ast)$ by given $(Q_ih)(t)=P_i(h(t))$ for each $i\in I$. It is easy to that $\Vert Q_i\Vert\leq1$ and $\ker(Q_i)=L_p(Y_i)^{\perp}=L_p({Y_i}^\perp)$ for all $i\in I$. For any $\varepsilon>0$, $h\in B_{L_{p'}(X^\ast)}$ with
\[
\Vert h\Vert-\Vert Q_ih\Vert<\sigma(\varepsilon)
\]
where $\sigma(\varepsilon)=\frac{\delta_{\{Y_i\}}(\varepsilon/2^{1/p'})^{p'}\varepsilon^{p'}}{2^{p'+1}p'}$. Put
\[
A=\big\{t\in[0,1]:\Vert h(t)\Vert-\Vert P_i(h(t))\Vert\geq\delta_{\{Y_i\}}(\varepsilon/2^{1/p'})\Vert h(t)\Vert\big\}.
\]
Note that
\[
\int_A\big(\delta_{\{Y_i\}}(\varepsilon/2^{1/p'})\Vert h(t)\Vert\big)^{p'}dt\leq\int_A\big(\Vert h(t)\Vert-\Vert P_i(h(t))\Vert\big)^{p'}dt
\]
\[
\leq\int_0^1\big(\Vert h(t)\Vert-\Vert P_i(h(t))\Vert\big)^{p'}dt\leq\Vert h\Vert^{p'}-\Vert Q_ih\Vert^{p'}\leq p'\big(\Vert h\Vert-\Vert Q_ih\Vert\big)<p'\sigma(\varepsilon).
\]
we obtain
\[
\int_A\Vert h(t)\Vert^{p'}\leq\varepsilon^{p'}/2^{p'+1}.
\]
Therefore,
\[
\Vert h-Q_ih\Vert=\left(\int_A\Vert h(t)-P_i(h(t))\Vert^{p'}dt+\int_{[0,1]\setminus A}\Vert h(t)-P_i(h(t))\Vert^{p'}dt\right)^{1/p'}
\]
\[
\leq\left(\int_A\big(2\Vert h(t)\Vert\big)^{p'}dt+\int_{[0,1]\setminus A}\big(\varepsilon/2^{1/p'}\Vert h(t)\Vert\big)^{p'}dt\right)^{1/p'}<\left(\varepsilon^{p'}/2+\varepsilon^{p'}/2\right)^{1/p'}=\varepsilon.
\]
We conclude that  $\{L_p(Y_i)\}_{i\in I}$ is a family of closed subspaces of $L_p(X)$ satisfying $(\ast)$-condition by Theorem \ref{3.5}.
\end{itemize} 
\end{proof}

The following lemma idea comes from \cite[Lem. 3.3]{Lancien2007}.
\begin{lemma}\label{3.8}
	Let $X$ be a Banach space and $Y$ be a closed subspace of $X$ satisfying $(\ast)$-condition. Let $\alpha$ be an ordinal and $\varepsilon>0$. If $\varphi\in s_{3\varepsilon}^\alpha(B_{X^\ast})$ and $\Vert\varphi|_Y\Vert>1-\delta_{\{Y\}}(\varepsilon)$, then $\varphi|_Y\in s_\varepsilon^\alpha(B_{Y^\ast})$.
\end{lemma}
\begin{proof}
We denote, for any ordinal $\alpha$, by $(H_\alpha)$ the implication to be proved. This will be done by transfinite induction on $\alpha$. $(H_0)$ is trivially true. For arbitrary an ordinal $\alpha$, we assume for all ordinal $\beta<\alpha$, $(H_\beta)$ is true. It also passes easily to limit ordinal. So we let $\alpha=\beta+1$ and assume now that it is true for $\beta$. Since $\Vert\varphi|_Y\Vert>1-\delta(\varepsilon)$, there is $y\in B_Y$ such that $|\varphi(y)|>1-\delta(\varepsilon)$. For any $w^\ast$-neighborhood $V$ of $\varphi|_Y$, we let
	\[
		W=\{\phi\in X^\ast:\phi|_Y\in V~{\rm and}~|\phi(y)|>1-\delta(\varepsilon)\}.
	\]
Then $W$ is a $w^\ast$-neighborhood of $\varphi$, so ${\rm diam}(W\cap s_{3\varepsilon}^\beta(B_{X^\ast}))>3\varepsilon$. Now, we let $P$ is the projection on $X^\ast$ with the property in Definition \ref{3.6} respect to $\{Y\}$. For any $\phi,\psi\in W\cap s_{3\varepsilon}^\beta(B_{X^\ast})$, then $\Vert\phi|_Y\Vert=\Vert P\phi\Vert,\Vert\psi|_Y\Vert=\Vert P\psi\Vert >1-\delta(\varepsilon)$. By the induction hypothesis, $\phi|_Y,\psi|_Y\in V\cap s_\varepsilon^\beta(B_{Y^\ast})$ and By $Y$ satisfying $(\ast)$-condition, $\Vert\phi-P\phi\Vert,\Vert\psi-P\psi\Vert<\varepsilon$. 
	\[
	\Vert\phi|_Y-\psi|_Y\Vert=\Vert P(\phi-\psi)\Vert\geq\Vert\phi-\psi\Vert-\Vert\phi-P\phi\Vert-\Vert\psi-P\psi\Vert>\Vert\phi-\psi\Vert-2\varepsilon.
	\]
	By the arbitrariness of $\phi$ and $\psi$, it follows that
	\[
	{\rm diam}(V\cap s_\varepsilon^\beta(B_{Y^\ast}))>\varepsilon.
	\]
	So $\varphi|_I\in s_\varepsilon^\alpha(B_{Y^\ast})$. This finishes our induction.
	\end{proof}
\begin{theorem}\label{3.9}
	Let $X$ be a Banach space and $Y$ be a closed subspace of $X$ satisfying $(\ast)$-condition. Then
	\[
	Sz(X)=\max\{Sz(Y),Sz(X/Y)\}.
	\] 
\end{theorem}
\begin{proof}	
If $\max\{Sz(Y),Sz(X/Y)\}=\infty$ is trivial by Proposition \ref{3.1} (b). Therefore, we assume that
\[
\max\{Sz(Y),Sz(X/Y)\}=\omega^\alpha
\]
for some ordinal $\alpha$. By Theorem \ref{3.5}, we let $P$ is the projection on $X^\ast$ with $\ker P=Y^\perp$ such that for any $\varepsilon>0$, there exist $\delta(\varepsilon)\in(0,1)$, such that for any $\varphi\in B_{X^\ast}$ with $\Vert\varphi\Vert-\Vert P\varphi\Vert<\delta(\varepsilon)$, then $\Vert\varphi-P\varphi\Vert<\varepsilon$. By natural isomorphism $(X/Y)^\ast\cong Y^\perp$, it is easy to verify that this isomorphism is $w^\ast$-$w^\ast$ homeomorphism and linearly isometric onto. so
	\[
	Sz(X/Y)=Sz(B_{Y^\perp})\leq\omega^\alpha.
	\]
For any $\varepsilon>0$, there exist $\beta<\alpha$ and $r\in\mathbb N$ such that $Sz_{3\varepsilon}(2B_{Y^\perp}),Sz_\varepsilon(Y)\leq\omega^\beta r$. Next, we will inductively prove that 
\[
s_{12\varepsilon}^{\omega^\beta(2mr)}\big(B_{X^\ast}\big)\subset2B_{Y^\perp}+(1-\delta(\varepsilon))^mB_{X^\ast}
\]
for all $m\in\mathbb N$. For $m=0$ it is trivially true. Assume that the conclusion holds for $m$, we must deduce that the conclusion is also true for $m+1$.
\[
s_{12\varepsilon}^{\omega^\beta(2(m+1)r)}\big(B_{X^\ast}\big)=s_{12\varepsilon}^{\omega^\beta(2r)}\Big(s_{12\varepsilon}^{\omega^\beta(2mr)}(B_{X^\ast})\Big).
\]
By the inductive hypothesis and Proposition \ref{3.1} (e),
\[
s_{12\varepsilon}^{\omega^\beta(2(m+1)r)}\big(B_{X^\ast}\big)\subset s_{12\varepsilon}^{\omega^\beta(2r)}\big(2B_{Y^\perp}+(1-\delta(\varepsilon))^mB_{X^\ast}\big)
\]
\[
\subset\bigcup_{i+j=2r}\big[s_{3\varepsilon}^{\omega^\beta i}\big(2B_{Y^\perp}\big)+s_{3\varepsilon}^{\omega^\beta j}\big((1-\delta(\varepsilon))^mB_{X^\ast}\big)\big]\subset2B_{Y^\perp}+s_{3\varepsilon}^{\omega^\beta r}\big((1-\delta(\varepsilon))^mB_{X^\ast}\big)
\]
\[
\subset2B_{Y^\perp}+(1-\delta(\varepsilon))^ms_{3\varepsilon/(1-\delta(\varepsilon))^m}^{\omega^\beta r}\big(B_{X^\ast}\big)\subset2B_{Y^\perp}+(1-\delta(\varepsilon))^ms_{3\varepsilon}^{\omega^\beta r}\big(B_{X^\ast}\big).
\]
For any $\varphi\in s_{12\varepsilon}^{\omega^\beta(2(m+1)r)}\big(B_{X^\ast}\big)$, then there exist $\psi\in s_{3\varepsilon}^{\omega^\beta r}\big(B_{X^\ast}\big)$ such that $\varphi|_Y=(1-\delta(\varepsilon))^m\psi|_Y$. It is easy to see that $\Vert\psi|_Y\Vert\leq1-\delta(\varepsilon)$. Indeed, if $\Vert\psi|_Y\Vert>1-\delta(\varepsilon)$, then $\psi|_Y\in s_\varepsilon^{\omega^\beta r}\big(B_{Y^\ast}\big)$ by Lemma \ref{3.6}. Which contradicts $Sz_\varepsilon(Y)\leq\omega^\beta r$. Therefore,
\[
\Vert P\varphi\Vert=\Vert\varphi|_Y\Vert=\Vert(1-\delta(\varepsilon))^m\psi|_Y\Vert\leq(1-\delta(\varepsilon))^{m+1}.
\]
Note that
 \[
 \varphi=(\varphi-P\varphi)+P\varphi\in2B_{Y^\perp}+(1-\delta(\varepsilon))^{m+1}B_{X^\ast}.
 \]
So
\[
s_{12\varepsilon}^{\omega^\beta(2(m+1)r)}\big(B_{X^\ast}\big)\subset2B_{Y^\perp}+(1-\delta(\varepsilon))^{m+1}B_{X^\ast}
\]
This finishes our induction. Fix a $m_0$ such that $(1-\delta(\varepsilon))^{m_0}<\varepsilon$, then
\[
s_{12\varepsilon}^{\omega^\beta(2m_0r)}\big(B_{X^\ast}\big)\subset2B_{Y^\perp}+\varepsilon B_{X^\ast}.
\]
By Proposition \ref{3.1} (e),
\[
s_{12\varepsilon}^{\omega^\beta(2m_0r+r)}\big(B_{X^\ast}\big)\subset s_{12\varepsilon}^{\omega^\beta r}\big(2B_{Y^\perp}+\varepsilon B_{X^\ast}\big)\subset\bigcup_{i+j=r}\big[s_{3\varepsilon}^{\omega^\beta i}\big(2B_{Y^\perp}\big)+s_{3\varepsilon}^{\omega^\beta j}\big(\varepsilon B_{X^\ast}\big)\big]
\]
\[
\subset s_{3\varepsilon}^{\omega^\beta r}\big(2B_{Y^\perp}\big)+\varepsilon B_{X^\ast}=\emptyset.
\]
This shows that $Sz_{12\varepsilon}(B_{X^\ast})<\omega^{\beta+1}\leq\omega^\alpha$. Therefore,
\[
	Sz(X)\leq\omega^\alpha=\max\{Sz(Y),Sz(X/Y)\}.
\]
which combined with Proposition \ref{3.1} (b), we can obtain that
	\[
	Sz(X)=\max\{Sz(Y),Sz(X/Y)\}.
	\]
\end{proof}
\begin{remark}
The research of the three space properties for Szlenk index of Banach spaces, that can see \cite{BL}, \cite{Causey2015} and \cite{Causey2020}.
\end{remark}
\begin{lemma}\label{3.11}
Let $X$ be a Banach space and $\{Y_i\}_{i\in I}$ be a family of subspaces of $X$ satisfying $(\ast)$-condition and $Sz(Y_i)\leq\omega^\alpha$ for all $i\in I$ and some ordinal $\alpha$. If
\[
\overline{{\bigcup}_{i\in I}Y_i}=X,
\]
then $Sz(X)\leq\omega^{\alpha+1}$.
\end{lemma}
\begin{proof}
For any $\varepsilon>0$ and any $\varphi\in s_{\varepsilon}^{\omega^\alpha}(B_{X^\ast})$, if $\Vert\varphi\Vert>1-\delta_{\{Y_i\}}(\varepsilon/3)$, then there exists $Y_{i_0}$ such that $\Vert\varphi|_{Y_{i_0}}\Vert>1-\delta_{\{Y_i\}}(\varepsilon/3)$. By Lemma \ref{3.6}, $\varphi|_{Y_{i_0}}\in s_{\varepsilon/3}^{\omega^\alpha}(B_{{Y_{i_0}}^\ast})$, which contradicts $Sz(Y_{i_0})\leq\omega^\alpha$. Therefore, $s_{\varepsilon}^{\omega^\alpha}(B_{X^\ast})\subset(1-\delta_{\{Y_i\}}(\varepsilon/3))B_{X^\ast}$, it follows that $Sz(X)\leq\omega^\alpha\cdot\omega=\omega^{\alpha+1}$ by Proposition \ref{3.1} (f).
\end{proof}

\begin{theorem}\label{3.12}
Let $X$ be a Banach space, there is a family of closed subspaces $\{Y_\xi\}_{\xi\leq\alpha}$ of $X$ satisfying $(\ast)$-condition, indexed by the ordinals $\xi$, $0\leq\xi\leq\alpha$, such that $Y_{\xi_1}\subset Y_{\xi_2}$ if $\xi_1<\xi_2\leq\alpha$, such that $Y_0=\{0\}$ and $Y_\alpha=X$, and such that $Y_\xi=[\cup_{\zeta<\xi}Y_{\zeta}]^{-}$ for each limit ordinal $\xi$. If $\alpha<\omega^{1+\beta}$ and $Sz(Y_{\xi+1}/Y_\xi)\leq\omega^\gamma$ for each $\xi$, $0\leq\xi<\alpha$, then $Sz(X)\leq\omega^{\gamma+\beta}$.
\end{theorem}
\begin{proof}
We will be done by transfinite induction on $\beta$. If $\beta=0$, then there is $m\in\mathbb N$ such that $\alpha<m$. By Theorem \ref{3.9}, we can inductively show that $Sz(X)\leq\omega^\gamma$.

It passes easily to limit ordinal.
	 
Now we let $\beta=\rho+1$, where $\beta\geq1$, and assume that it is true for $\rho$. If $\alpha<\omega^{1+\rho}$, then $Sz(X)\leq\omega^{\gamma+\rho}\leq\omega^{\gamma+\beta}$ by induction hypothesis. Therefore we assume $\alpha\geq\omega^{1+\rho}$, since $
\alpha<\omega^{1+\beta}$, there is $n\in\mathbb N$ such that $\alpha<\omega^{1+\rho}\cdot n$. For any ordinals $\xi$ and $\eta$ with $\xi\leq\xi+\eta\leq\alpha$, we know that $\{Y_{\xi+\zeta}/Y_\xi\}_{\zeta\leq\eta}$ is a family of subspaces of $Y_{\xi+\eta}/Y_\xi$ which satisfying $(\ast)$-condition and 
\[
Sz\big((Y_{\xi+\zeta+1}/Y_\xi)/(Y_{\xi+\zeta}/Y_\xi)\big)=Sz(Y_{\xi+\zeta+1}/Y_{\xi+\zeta})\leq\omega^\gamma
\]
for each ordinal $\zeta$, $0\leq\zeta<\eta$. Using Theorem \ref{3.9}, it is easy to see that we only need show that
\[
Sz(Y_{\omega^{1+\rho}})\leq\omega^{\gamma+\beta}.
\]
Therefore, without loss of generality, we can assume that $\alpha=\omega^{1+\rho}$. It follows that
\[
\overline{{\bigcup}_{\xi<\omega^{1+\rho}}Y_\xi}=X.
\]
For any ordinal $\xi<\omega^{1+\rho}$, we know that $\{Y_\zeta\}_{\zeta\leq\xi}$ is a family of subspaces of $Y_\xi$ which has $``nice''$ property. By induction hypothesis we have $Sz(Y_\xi)\leq\omega^{\gamma+\rho}$. By Lemma \ref{3.11},
\[
Sz(X)\leq\omega^{\gamma+\rho+1}=\omega^{\gamma+\beta}.
\]
This finishes our induction.
\end{proof}
 
\begin{corollary}\label{3.13}
Let $X$ be a Banach space, there is a family of closed subspaces $\{Y_\xi\}_{\xi\leq\alpha}$ of $X$ satisfying $(\ast)$-condition, indexed by the ordinals $\xi$, $0\leq\xi\leq\alpha$, such that $Y_{\xi_1}\subset Y_{\xi_2}$ if $\xi_1<\xi_2\leq\alpha$, such that $Y_0=\{0\}$ and $Y_\alpha=X$, and such that $Y_\xi=[\cup_{\zeta<\xi}Y_{\zeta}]^{-}$ for each limit ordinal $\xi$. If $\alpha<\omega^{\beta}$ and $Sz(Y_{\xi+1}/Y_\xi)\leq\omega^{\gamma+1}$ for each $\xi$, $0\leq\xi<\alpha$, then $Sz(X)\leq\omega^{\gamma+\beta}$.
\end{corollary}
\begin{proof}
If $\beta<\omega$, then there is $n\in\mathbb N$ such that $\beta=1+n$. By Theorem \ref{3.12}
\[
Sz(X)\leq\omega^{(\gamma+1)+n}=\omega^{\gamma+(1+n)}=\omega^{\gamma+\beta}.
\]

If $\beta\geq\omega$, then $\beta=1+\beta$. By Theorem \ref{3.12}
\[
Sz(X)\leq\omega^{(\gamma+1)+\beta}=\omega^{\gamma+(1+\beta)}=\omega^{\gamma+\beta}.
\]
\end{proof}

\section{\large Szlenk and $w^\ast$-dentability indices of C*-algebras}
Given an ordinal $\xi$, we let $\Gamma(\xi)$ denote is the minimum ordinal number which is not less than $\xi$ of the form $\omega^\zeta$ for some $\zeta$. Since $\omega^\xi\geq\xi$ for any ordinal $\xi$, this minimum exists. For completeness, we agree that $\Gamma(\infty)=\omega\infty=\infty$.

Next, we introduce the result of R.M. Causey \cite{Causey2017,Causey2017b}.
\begin{theorem}[\text{R.M. Causey}]\label{4.1}
	Let $K$ be a compact, Hausdorff topological space and $1<p<\infty$. Then 
	\[
	Sz(C(K))=\Gamma(i(K))~and~Dz(C(K))=Sz(L_p(C(K)))=\omega Sz(C(K))=\omega\Gamma(i(K)).
	\]
\end{theorem}
In fact, M. Causey's result has already provided the computation of Szlenk and $w^\ast$-dentability indices on commutative C*-algebras, and here we simply generalize it to the $C_0(\Omega)$ with $\Omega$ is an infinite, locally compact, Hausdorff topological space.

Given an ordinal $\xi$, we let $\Gamma'(\xi)$ denote is the minimum ordinal number which is greater than $\xi$ of the form $\omega^\zeta$ for some $\zeta$, it's easy to see that $\zeta$ is a successor ordinal when $\xi>0$. Since $\omega^{\xi+1}>\xi$ for any ordinal $\xi$, this minimum exists. For completeness, we agree that $\Gamma'(\infty)=\infty$.
\begin{corollary}
Let $\Omega$ be an infinite, locally compact, Hausdorff topological space and $1<p<\infty$. Then
\begin{equation}\label{4.1s}
Sz(C_0(\Omega))=\Gamma'(i(\Omega))~and~Dz(C_0(\Omega))=Sz(L_p(C_0(\Omega)))=\omega Sz(C_0(\Omega))=\omega\Gamma'(i(\Omega)).
\end{equation}	
\end{corollary}
\begin{proof}
If $\Omega$ is compact. Then $C_0(\Omega)=C(\Omega)$ and $i(\Omega)$ is never limit ordinal, so $\Gamma(i(\Omega))=\Gamma'(i(\Omega))$. By Theorem \ref{4.1}, it follows that
\[
Sz(C_0(\Omega))=\Gamma'(i(\Omega))~and~Dz(C_0(\Omega))=Sz(L_p(C_0(\Omega)))=\omega Sz(C_0(\Omega))=\omega\Gamma'(i(\Omega)).            
\]

If $\Omega$ is not compact, we put $\tilde{\Omega}=\Omega\cup\{\infty\}$ is one-point compactification of $\Omega$. It is easy to see that
\[         
\Omega^{(\alpha)}=\tilde\Omega^{(\alpha)}\backslash\{\infty\}
\]
for any ordinal number $\alpha$. So $i(\Omega)\leq i\big(\tilde\Omega\big)\leq i(\Omega)+1$. By $i\big(\tilde\Omega\big)$ is a successor ordinal, we can see that $\Gamma\big(i\big(\tilde\Omega\big)\big)=\Gamma'\big(i\big(\tilde\Omega\big)\big)=\Gamma'(i(\Omega))$. Note that $C\big(\tilde\Omega\big)\approx C_0(\Omega)\oplus\mathbb C$ or $C\big(\tilde\Omega\big)\approx C_0(\Omega)\oplus\mathbb R$. By Proposition \ref{3.1} (e) and Theorem \ref{4.1}, we have
\[
Sz(C_0(\Omega))=Sz\big(C\big(\tilde{\Omega}\big)\big)=\Gamma\big(i\big(\tilde\Omega\big)\big)=\Gamma'(i(\Omega))
\]
and
\[
Dz(C_0(\Omega))=Dz\big(C\big(\tilde{\Omega}\big)\big)=Sz\big(L_p\big(C\big(\tilde{\Omega}\big)\big)\big)=\omega\Gamma\big(i\big(\tilde\Omega\big)\big)=\omega\Gamma'(i(\Omega)),
\]
\[
Sz\big(L_p\big(C(\tilde\Omega\big)\big)\big)=\max\{Sz(L_p(C_0(\Omega))),Sz(L_p)\}=Sz(L_p(C_0(\Omega))).
\]
\end{proof}

Next, we will deal with the Szlenk and $w^\ast$-dentability indices of C*-algebra. First, a natural idea is to consider whether the noncommutative version of equation (\ref{4.1s}) holds. Because $\Omega$ is homeomorphic to $\{\delta_t:t\in\Omega\}=PS(C_0(\Omega))$, where $PS(C_0(\Omega))$ equipped with the weak* topology. Comparing with equation (\ref{4.1s}), one guess is that $Sz(\mathcal A)=\Gamma'(i(PS(\mathcal A)))$ for any infinite dimension C*-algebra $\mathcal A$. However, the following theorem shows that this guess is incorrect.

\begin{theorem}
	Let $\mathcal A$ be a C*-algebra. Then $i(PS(\mathcal A))<\infty$ (i.e., $PS(\mathcal A)$ is scattered), where $PS(\mathcal A)$ equipped with the weak* topology, if and only if $\mathcal A$ is a commutative scattered C*-algebra.
\end{theorem}
\begin{proof}
	Sufficiency. It is obviously.
	
	Necessity. We let $\{H,\pi\}$ is an any nonzero irreducible representation, for all $\xi\in H$ and $\Vert\xi\Vert=1$, we let $\omega_\xi=(\pi(\cdot)\xi|\xi)$. By $\xi$ is a cyclic vector, then $\{H,\pi\}$ and $\{H_{\omega_\xi},\pi_{\omega_\xi}\}$ are unitarily equivalent, so $\omega_\xi\in PS(\mathcal A)$. We let $\Omega=\{\omega_\xi:\xi\in H~\rm{and}~\Vert\xi\Vert=1\}$. If $\dim(H)>1$, it is easy to see $\Omega'=\Omega$, where $PS(\mathcal A)$ equipped with the weak* topology. So for any ordinal $\alpha$, we have 
	\[
	PS(\mathcal A)^{(\alpha)}\supset\Omega.
	\]
It follows that $i(PS(\mathcal A))=\infty$ (i.e., $PS(\mathcal A)$ is not scattered). This is a contradiction. So $\dim(H)=1$, this yields that $\mathcal A$ is commutative, and $PS(\mathcal A)$ is homeomorphic to $\widehat{\mathcal A}$. Therefore $\mathcal A$ is scattered.
\end{proof}
\begin{remark}
There exists C*-algebra $\mathcal A$ such that $PS(\mathcal A)$ may not be $w^\ast$-closed, see for example a paper of J. Glimm \cite{Glimm}. So of course it may be not $w^\ast$-compact.
\end{remark}
Therefore, we need to look for alternative replacements to analogize $i(\Omega)$. Inspired by reference S. Ghasemi and P. Koszmide \cite{GK}. Suppose that $\mathcal A$ is a C*-algebra. and for every ordinal number $\alpha$, we define the ideal $\mathcal A^{(\alpha)}$ by transfinite induction: $\mathcal A^{(0)}=\{0\}$; if $\alpha=\beta+1$, by Theorem \ref{2.7}, we denote $\mathcal A^{(\alpha)}$ is the unique ideal $\mathcal A^{(\beta+1)}\supset\mathcal A^{(\beta)}$ such that
\[
\mathcal A^{(\beta+1)}/\mathcal A^{(\beta)}=\mathcal I^{At}(\mathcal A/\mathcal A^{(\beta)}).
\]
If $\alpha$ is a limit ordinal, we denote
\[
\mathcal A^{(\alpha)}=\overline{\bigcup_{\beta<\alpha}\mathcal A^{(\beta)}}~~~~~~~(\rm{norm~clusre}).
\]
We let
\[
i(\mathcal A)=\min\{\alpha:\mathcal A^{(\alpha)}=\mathcal A\}
\]
if this class of ordinals is non-empty, and we write $i(\mathcal A)=\infty$ otherwise. We agree to the convention that $\xi<\infty$ for any ordinal $\xi$.

The following theorem follows from S. Ghasemi and P. Koszmide \cite{GK}.
\begin{theorem}[\text{S. Ghasemi and P. Koszmide}]\label{4.5}
Let $\mathcal A$ be a C*-algebra. Then $i(\mathcal A)<\infty$ if and only if $\mathcal A$ is scattered.
\end{theorem}
In fact, $i(\mathcal A)=i\big(\widehat{\mathcal A}\big)$ when $\mathcal A$ is a scattered C*-algebra (i.e., $i(\mathcal A)<\infty$), that can see by the following theorem:
\begin{theorem}\label{4.6}
	Let $\mathcal A$ be a separable or type I C*-algebra. Then for any ordinal $\alpha$,
	\[
	\mathcal A^{(\alpha)}=\ker'\left(\widehat{\mathcal A}^{(\alpha)}\right).
	\]
	In particular, let $\Omega$ be a locally compact, Hausdorff topological space. Then for any ordinal $\alpha$,
	\[
	C_0(\Omega)^{(\alpha)}=\{f\in C_0(\Omega):f|_{\Omega^{(\alpha)}}=0\}.
	\]
\end{theorem}
\begin{proof}
	We denote, for any ordinal $\alpha$, by $(H_\alpha)$ the  implication to be proved. This will be done by transfinite induction on $\alpha$. $(H_0)$ is trivially true. For arbitrary an ordinal $\alpha$, we assume for all ordinal $\beta<\alpha$, $(H_\beta)$ is true. If $\alpha$ is a limit ordinal, then
	\[
	\rm{hull}'\left(\mathcal A^{(\alpha)}\right)=\rm{hull}'\left(\overline{\cup_{\beta<\alpha}\mathcal A^{(\beta)}}\right)=\bigcap_{\beta<\alpha}\rm{hull}'\left(\mathcal A^{(\beta)}\right)=\bigcap_{\beta<\alpha}\rm{hull}'\left(\ker'\left(\widehat{\mathcal A}^{(\beta)}\right)\right)=\bigcap_{\beta<\alpha}\widehat{\mathcal A}^{(\beta)}=\widehat{\mathcal A}^{(\alpha)},
	\]
	so
	\[
	\mathcal A^{(\alpha)}=\ker'\left(\widehat{\mathcal A}^{(\alpha)}\right).
	\]
	If $\alpha=\beta+1$. First, we claim that $\mathcal I^{At}(\mathcal A)=\mathcal A^{(1)}=\ker'\left(\widehat{\mathcal A}^{~'}\right)$. We put $J_1=\ker'\left(\widehat{\mathcal A}^{~'}\right)$, since $\widehat{J_1}\simeq\widehat{\mathcal A}\backslash{\rm hull}'(J_1)=\widehat{\mathcal A}\backslash\widehat{\mathcal A}^{~'}$ is discrete topology space, by Theorem \ref{2.7} and Theorem \ref{2.9}, $J_1\subset\mathcal I^{At}(\mathcal A)=\mathcal A^{(1)}$. Conversely, $\mathcal A^{(1)}=\mathcal I^{At}(\mathcal A)$ is isomorphic to the $c_0$-sum of $\{\mathcal K(H_i)\}_{i\in I}$, where $\{H_i\}_{i\in I}$ is a family of Hilbert spaces. So $\widehat{\mathcal A^{(1)}}\simeq\widehat{\mathcal A}\backslash\rm{hull}'(\mathcal A^{(1)})$ is discrete. It follows that $\widehat{\mathcal A}^{~'}\subset\rm{hull}'(\mathcal A^{(1)})$,
	\[
	\ker'(\widehat{\mathcal A}^{~'})\supset\ker'(\rm{hull}'(\mathcal A^{(1)}))=\mathcal A^{(1)}.
	\] 
	Therefore, $\mathcal I^{At}(\mathcal A)=\mathcal A^{(1)}=\ker'\left(\widehat{\mathcal A}^{'}\right)$. Finally, 
	\[
	\mathcal A^{(\beta+1)}/\mathcal A^{(\beta)}=\mathcal I^{At}(\mathcal A/\mathcal A^{(\beta)})=\ker'\left(\widehat{\mathcal A/\mathcal A^{(\beta)}}'\right)=\ker'\left(\left(\rm{hull}'\left(\mathcal A^{(\beta)}\right)\right)'\right)\Big/\mathcal A^{(\beta)}
	\]
	\[
	=\ker'\left(\left(\rm{hull}'\left(\ker'\left(\widehat{\mathcal A}^{(\beta)}\right)\right)\right)'\right)\Big/\mathcal A^{(\beta)}=\ker'\left(\left(\widehat{\mathcal A}^{(\beta)}\right)'\right)\Big/\mathcal A^{(\beta)}
	\]
	\[
	=\ker'\left(\widehat{\mathcal A}^{(\beta+1)}\right)\Big/\mathcal A^{(\beta)},
	\]
    so
    \[
     \mathcal A^{(\alpha)}=\ker'\left(\widehat{\mathcal A}^{(\alpha)}\right).
    \]
	\end{proof}
\begin{corollary}\label{4.7}
	Let $\mathcal A$ be a separable or type I C*-algebra. Then
	\[
	i(\mathcal A)=i\big(\widehat{\mathcal A}\big).
	\]
	In particular, let $\Omega$ be a locally compact, Hausdorff topological space, then
	\[
	i(C_0(\Omega))=i(\Omega).
	\]
\end{corollary}
\begin{proof}
	
	We know that for any ordinal $\alpha$, then $\widehat{\mathcal A}^{(\alpha)}=\emptyset$ if and only if
	$\ker'(\widehat{\mathcal A}^{(\alpha)})=\mathcal A$, if and only if $\mathcal A^{(\alpha)}=\mathcal A$ by Theorem \ref{4.6}. So
	\[
	i(\mathcal A)=i\big(\widehat{\mathcal A}\big).
	\]
\end{proof}

\noindent{\bf The up estimates}

Let $H$ be a Hilbert space. For any finite dimensional subspaces $K$ and $L$ of $H$, we denote
\[
Y_{K,L}={\rm span}\big(\{(\cdot|e)f:e\in K, f\in L\}\big)\subset\mathcal K(H).
\]
\begin{lemma}\label{4.8}
Let $H$ be a Hilbert space and $1<p<\infty$. Then
\begin{itemize}
\item [(i)]
$\{Y_{K,L}:K~{\rm and}~L~{\rm are~finite~dimensional~subspaces~of}~H\}$ is a family of closed subspaces of $\mathcal K(H)$ satisfying $(\ast)$-condition.
\item [(ii)]
$\{L_p(Y_{K,L}):K~{\rm and}~L~{\rm are~finite~dimensional~subspaces~of}~H\}$ is a family of closed subspaces of $L_p(\mathcal K(H))$ satisfying $(\ast)$-condition.
\end{itemize}
\end{lemma}
\begin{proof}
We know $\mathcal S_1(H)$, trace-class operators on $H$, is that dual of $\mathcal K(H)$ up to the action by trace (e.g. see the book \cite[Th. 4.2.1]{Murphy}). For any finite dimensional subspaces $K$ and $L$ of $H$, it is easy to see that
\[
{Y_{K,L}}^\perp=\{u\in\mathcal S_1(H):P_LuP_K=0\}.
\]
We define a projection $P_{K,L}$ on $\mathcal S_1(H)$ by given $u\mapsto P_LuP_K$. It is obvious that $\ker(P_{K,L})={Y_{K,L}}^\perp$.
For any $\varepsilon>0$, we put $\delta(\varepsilon)=\varepsilon^2/8$. If $u\in B_{\mathcal S_1(H)}$ with $\Vert u\Vert_1$-$\Vert P_{K,L}u\Vert_1<\varepsilon^2/8$. 
\[
\Vert uP_{K^\perp}\Vert_1\leq\left(\Vert u\Vert_1^2-\Vert uP_K\Vert_1^2\right)^{\frac{1}{2}}\leq\big((\Vert u\Vert_1+\Vert uP_K\Vert_1)(\Vert u\Vert_1-\Vert P_LuP_K\Vert_1)\big)^{\frac{1}{2}}<\varepsilon/2,
\]
\big(there exist $u_1,u_2\in\mathcal S_2(H)$ such that $u=u_1u_2$ and $\Vert u\Vert_1=\Vert u_1\Vert_2\Vert u_2\Vert_2$, then $\Vert u\Vert_1=\Vert u_1\Vert_2\Vert u_2\Vert_2=\Vert u_1\Vert_2\left(\Vert u_2P_K\Vert_2^2+\Vert u_2P_{K^\perp}\Vert_2^2\right)^{\frac{1}{2}}\geq\left(\Vert uP_K\Vert_1^2+\Vert uP_{K^\perp}\Vert_1^2\right)^{\frac{1}{2}}$\big). In the same way, $\Vert P_{L^\perp}u\Vert_1<\varepsilon/2$.
\[
\Vert u-P_{K,L}u\Vert_1=\Vert u-P_LuP_K\Vert_1\leq\Vert P_{L^\perp}uP_K\Vert_1+\Vert uP_{K^\perp}\Vert_1<\varepsilon/2+\varepsilon/2=\varepsilon.
\]
Therefore, (i) claim holds by Theorem \ref{3.5}. By Proposition \ref{3.7} (c), (ii) claim holds.
\end{proof}

\begin{lemma}\label{4.9}
Let $H$ be a Hilbert space and $1<p<\infty$. Then
\[
Sz(\mathcal K(H))\leq\omega,~Sz(L_p(\mathcal K(H)))\leq\omega^2.
\]
\end{lemma}
\begin{proof}
We put $I=\{(K,L):K~{\rm and}~L~{\rm are~finite~dimensional~subspaces~of}~H\}$. By Lemma \ref{4.8} $\{Y_{K,L}\}_{(K,L)\in I}$ is a family of closed subspaces of $\mathcal K(H)$ satisfying $(\ast)$-condition. It is easy to see that
\[
\overline{{\bigcup}_{(K,L)\in I}Y_{K,L}}=\mathcal K(H).
\]
Since $\dim(Y_{K,L})<\infty$ for all $(K,L)\in I$, it follows that $Sz(Y_{K,L})=1=\omega^0$ for all $(K,L)\in I$. Using Lemma \ref{3.11} we can obtain
\[
Sz(\mathcal K(H))\leq\omega.
\]

On the other hand, using the density of the simple functions in $L_p(\mathcal K(H))$ and Lemma \ref{4.8}, we can obtain $\{L_p(Y_{K,L})\}_{(K,L)\in I}$ is a family of closed subspaces of $L_p(\mathcal K(H))$ satisfying $(\ast)$-condition. and
\[
\overline{{\bigcup}_{(K,L)\in I}L_p(Y_{K,L})}=L_p(\mathcal K(H)).
\]
$Sz(L_p(Y_{K,L})=Sz(L_p)=\omega$ for each $(K,L)\in I$, since $L_p$ is uniformly smooth. By Lemma \ref{3.9},
\[
Sz(L_p(\mathcal K(H)))\leq\omega^2.
\]
\end{proof}
\begin{remark}
We know that $\mathcal K(H)\simeq H\check{\otimes}H^\ast$, where $H\check{\otimes}H^\ast$ is the injective tensor product of $H$ and $H^\ast$. By the Causey's result \cite[Th. 1.3]{Causey2017}, it is easy obtain that $Sz(\mathcal K(H))\leq\omega$. Using the result of \cite{HS,Causey2022a}, it follows that $Sz(L_p(\mathcal K(H)))\leq\omega^2$. These results are quite powerful, here I would like to give a elementary proof.
\end{remark}
 
\begin{lemma}\label{4.11}
Let $\mathcal A$ be a C*-algebra and $1<p<\infty$. If $i(\mathcal A)<\omega^\beta$, then
\[
Sz(\mathcal A)\leq\omega^\beta,~Sz(L_p(\mathcal A))\leq\omega^{1+\beta}.
\]
\end{lemma}
\begin{proof}
For any $\xi\leq i(\mathcal A)$, we let $\overline{\mathcal A^{(\xi)}}^{w^\ast}$ is weak* closure in $\mathcal A^{\ast\ast}$. It is easy to see that $\overline{\mathcal A^{(\xi)}}^{w^\ast}$ is a $w^\ast$-closed ideal of $\mathcal A^{\ast\ast}$. Hence there exist a central projection $z_\xi$ of $\mathcal A^{\ast\ast}$ such that $\overline{\mathcal A^{(\xi)}}^{w^\ast}=\mathcal A^{\ast\ast}z_\xi$. Classical conclusion is that $\overline{\mathcal A^{(\xi)}}^{w^\ast}=\mathcal A^{(\xi)\perp\perp}$, so $\mathcal A^{(\xi)\perp}={^\perp(\mathcal A^{(\xi)\perp\perp})}=\mathcal A^\ast(1-z_\xi)$, where $1$ is the unit of $\mathcal A^{\ast\ast}$. We denote a projection $P_\xi$ on $\mathcal A^\ast$ with $\ker(P_\xi)=\mathcal A^\ast(1-z_\xi)=\mathcal A^{(\xi)\perp}$ by given $\varphi\in\mathcal A^\ast\mapsto\varphi z_\xi$. For any $\varphi\in\mathcal A^\ast$ and $\varepsilon>0$, if $\Vert\varphi\Vert-\Vert P_\xi\varphi\Vert<\varepsilon$. Then
\[
\Vert\varphi-P_\xi\varphi\Vert=\Vert=\Vert\varphi(1-z_\xi)\Vert=\Vert\varphi\Vert-\Vert\varphi z_\xi\Vert=\Vert\varphi\Vert-\Vert P_\xi\varphi\Vert<\varepsilon.
\]
This implies that $\{\mathcal A^{(\xi)}\}_{\xi\leq\alpha}$ is a family of closed ideals of $\mathcal A$ satisfying $(\ast)$-condition, where $\alpha=i(\mathcal A)$. By Theorem \ref{2.7} (b) and Lemma \ref{4.9}, $Sz\left(\mathcal A^{({\xi+1})}/\mathcal A^{(\xi)}\right)\leq\omega$ for each $\xi<\alpha$. Therefore, using Corollary \ref{3.13} we can obtain
\[
Sz(\mathcal A)\leq\omega^\beta.
\]

On the other hand, for any limit ordinal $\xi\leq\alpha$, then
\[
L_p(\mathcal A^{(\xi)})=\overline{{\bigcup}_{\zeta<\xi}L_p(\mathcal A^{(\zeta)})}
\]
by the density of the simple functions in $L_p(\mathcal A^{(\xi)})$. For each $\xi<\alpha$,
\[
Sz\left(L_p(\mathcal A^{({\xi+1})})/L_p(\mathcal A^{(\xi)})\right)=Sz\left(L_p\left(\mathcal A^{({\xi+1})}/\mathcal A^{(\xi)}\right)\right)\leq\omega^2
\]
by Lemma \ref{4.9}. Using the Proposition \ref{3.7} (c) and Corollary \ref{3.13} we can obtain
\[
Sz(L_p(\mathcal A))\leq\omega^{1+\beta}.
\]
\end{proof}

\noindent{\bf The lower estimates}

Let $\mathcal A$ be a C*-algebra. The set of pure states on $\mathcal A$ is denoted by $PS(\mathcal A)$. For any $\varphi\in PS(\mathcal A)$, we denote by $\{H_\varphi,\pi_\varphi\}$ the GNS reprensentation of $\varphi$. We put $\varPsi$ is that the canonical map $PS(\mathcal A)\to\widehat{\mathcal A}$ by setting $\varphi\mapsto[H_\varphi,\pi_\varphi]$. Then $\varPsi$ is a continuous and open map by \cite[Th. 3.4.11]{Dixmier}, where $PS(\mathcal A)$ equipped with the weak* topology.

We know that for every GNS representation $\{H_\varphi,\pi_\varphi\}$ there is an unique extension the normal representation $\{H_\varphi,\tilde{\pi_\varphi}\}$ of $\mathcal A^{\ast\ast}$ onto $\pi_\varphi(\mathcal A)''=\overline{\pi_\varphi(\mathcal A)}^{so}=\mathcal B(H_\varphi)$, (e.g. see \cite{Takesaki}). We denote by $z(\varphi)$ the central projection in $\mathcal A^{\ast\ast}$ such that $\ker(\tilde{\pi_\varphi})=\mathcal A^{\ast\ast}(1-z(\varphi))$, where $1$ is the unit of $\mathcal A^{\ast\ast}$. It is easy to see that $\varphi=\varphi z_\varphi\in\mathcal A^\ast z(\varphi)$. Supposed that $\varphi,\psi\in PS(\mathcal A)$. If $\{H_\psi,\pi_\psi\}$ and $\{H_\varphi,\pi_\varphi\}$ are unitarily equivalent. It is easy to see that $z(\varphi)=z(\psi)$. If $\{H_\psi,\pi_\psi\}$ and $\{H_\varphi,\pi_\varphi\}$ are not unitarily equivalent. Then We have $z(\varphi)\perp z(\psi)$. Indeed, note that $\tilde{\pi_\varphi}(z(\psi))$ is a projection and belongs to $\pi_\varphi(\mathcal A)'=\mathbb CI_{H_\varphi}$. It follows that $\tilde{\pi_\varphi}(z(\psi))=0$ or $I_{H_\varphi}$. If $\tilde{\pi_\varphi}(z(\psi))=0$, then $z(\varphi)\perp z(\psi)$ obviously. If $\tilde{\pi_\varphi}(z(\psi))=I_{H_\varphi}$, then $z(\varphi)-z(\psi)\in\ker(\tilde{\pi_\varphi})$. It follows that $(z(\varphi)-z(\psi))(1-z(\varphi))=z(\varphi)-z(\psi)$, so $z(\varphi)=z(\psi)z(\varphi)$. In the same way, $\tilde{\pi_\psi}(z(\varphi))=I_{H_\psi}$ and $z(\psi)=z(\varphi)z(\psi)$. It follows that $z(\varphi)=z(\psi)$, then $\ker(\tilde{\pi_\varphi})=\ker(\tilde{\pi_\psi})$. $\{\pi_\varphi\oplus\pi_\psi,H_\varphi\oplus H_\psi\}$ is a representation of $\mathcal A$, we know that there is a unique extension representation of $\mathcal A^{\ast\ast}$ onto $(\pi_\varphi\oplus\pi_\psi)(\mathcal A)''=\overline{(\pi_\varphi\oplus\pi_\psi)(\mathcal A)}^{so}$, it is easy to see that this extension representation is $\{\tilde{\pi_\varphi}\oplus\tilde{\pi_\psi},H_\varphi\oplus H_\psi\}$. Since $\{H_\psi,\pi_\psi\}$ and $\{H_\varphi,\pi_\varphi\}$ are not unitarily equivalent, it is easy to show that $(\pi_\varphi\oplus\pi_\psi)(\mathcal A)'=\mathbb C I_{H_\varphi}\oplus\mathbb C I_{H_\psi}$. Then $\{\tilde{\pi_\varphi}(a)\oplus\tilde{\pi_\psi}(a):a\in\mathcal A^{\ast\ast}\}=(\pi_\varphi\oplus\pi_\psi)(\mathcal A)''=\mathcal B(H_\varphi)\oplus\mathcal B(H_\psi)$. This contradicts the $\ker(\tilde{\pi_\varphi})=\ker(\tilde{\pi_\psi})$.
\begin{lemma}\label{4.12}
Let $\mathcal A$ be a type I C*-algebra and $\varphi\in PS(\mathcal A)$. If there is $[H,\pi]\in\widehat{\mathcal A}$ such that $\ker[H,\pi]\subsetneqq\ker[H_\varphi,\pi_\varphi]$. Then $\dim(H)=\infty$, and for any $w^\ast$-neighborhood $V$ of $\varphi$ and finite dimensional subspace $K$ of $H$, there exists a vector $e\in H$ with $\Vert e\Vert=1$ and $e\perp K$ such that $(\pi(\cdot)e|e)\in V$.
\end{lemma}
\begin{proof}
Note that
\[
\pi^{-1}\big(\pi(\ker[H_\varphi,\pi_\varphi])\big)=\ker[H_\varphi,\pi_\varphi]+\ker[H,\pi]=\ker[H_\varphi,\pi_\varphi]\subsetneqq\mathcal A,
\]
this implies that $\pi(\ker[H_\varphi,\pi_\varphi])$ is a nontrivial ideal of $\pi(\mathcal A)$. It follows that $\dim(H)=\infty$. We know that $\{H,\pi|_{\ker[H_\varphi,\pi_\varphi]}\}$ is a non-zero irreducible representation and $\ker[H_\varphi,\pi_\varphi]$ is type I by \cite[Th 5.6.2]{Murphy}, which yields $\mathcal K(H)\subset\pi([H_\varphi,\pi_\varphi])$. Now, we choose a $\delta\in(0,1)$ and a $w^\ast$-neighborhood $U$ of $\varphi$ such that $U+\delta B_{\mathcal A^\ast}\subset V$, and choose $a\in\pi([H_\varphi,\pi_\varphi])$ such that $\pi(a)=P_K$. Let $W=U\cap\{\phi\in\mathcal A^\ast:\vert\phi(a)-\varphi(a)\vert<\delta^2/16\}=U\cap\{\phi\in\mathcal A^\ast:\vert\phi(a)\vert<\delta^2/16\}$, it is a $w^\ast$-neighborhood of $\varphi$. We know that $\varPsi(W\cap PS(\mathcal A))$ is a neighborhood of $\varPsi(\varphi)$, then there exists a ideal $I$ of $\mathcal A$ such that $\varPsi(W\cap PS(\mathcal A))=\widehat{\mathcal A}\setminus{\rm hull}'(I)$. It is easy to see that $[H,\pi]\in\varPsi(W\cap PS(\mathcal A))$. Then there is $\phi\in W\cap PS(\mathcal A)$ such that $\varPsi(\phi)=[H,\pi]$. It is easy to show that there is $f\in H$ with $\Vert f\Vert=1$ such that $\phi=(\pi(\cdot)f|f)$. Then we have $\Vert P_Kf\Vert^2=\vert(\pi(a)f|f)\vert=\vert\phi(a)\vert<\delta^2/16$. Now we let $e=(f-P_Kf)/\Vert f-P_Kf\Vert$. Then
\begin{align*}
&\big\Vert(\pi(\cdot)f|f)-(\pi_(\cdot)e|e)\big\Vert\\
\leq~&\big\Vert(\pi(\cdot)(f-P_Kf)|P_Kf)\big\Vert+\big\Vert(\pi(\cdot)(P_Kf|f-P_Kf)\big\Vert\\
&+\big\Vert(\pi(\cdot)P_Kf|P_Kf)\big\Vert+\big(1-\Vert f-P_Kf\Vert^2\big)\big\Vert(\pi(\cdot)e|e)\big\Vert\\
<~&\delta/4+\delta/4+\delta/4+\delta/4=\delta.
\end{align*}
Therefore $(\pi(\cdot)e|e)\in\phi+\delta B_{\mathcal A^\ast}\subset U+\delta B_{\mathcal A^\ast}\subset V$ and $e\perp K$.
\end{proof}

The following lemma \ref{4.13} and \ref{4.15} can be seen as the non-commutative version of \cite[Lem 4.1]{Causey2017b}. For the proof of the main conclusion, we only need Lemma \ref{4.13}. If the reader is not interested, you can skip Lemma \ref{4.15} and will not affect the subsequent reading. The Lemma \ref{4.15} is a more refined promotion of the Lemma \ref{4.13}. 

\begin{lemma}\label{4.13}
	Let $\mathcal A$ be a C*-algebra. If an ordinal $\xi<i\big(\widehat{\mathcal A}\big)$.
\begin{itemize}
\item [(i)] For any $\varepsilon\in(0,2)$, then $\varPsi^{-1}\big(\widehat{\mathcal A}^{(\xi)}\big)\subset s_\varepsilon^\xi(B_{\mathcal A^\ast})$.
\item [(ii)] If $\mathcal A$ is type I. For any $\varepsilon\in(0,1)$, then $\varPsi^{-1}\big(\widehat{\mathcal A}^{(\xi)}\big)\subset d_\varepsilon^{\omega\xi}(B_{\mathcal A^\ast})$.
\end{itemize}
\end{lemma}
\begin{proof}
\begin{itemize}
\item [(i)]
We will be done by transfinite induction on $\xi$. It is trivially true for $\xi=0$. It also passes easily to limit ordinal. So we let $\xi=\zeta+1$ and assume now that it is true for $\zeta$. If $\varphi\in\varPsi^{-1}(\widehat{\mathcal A}^{(\xi)})\subset s_\varepsilon^\zeta(B_{\mathcal A^\ast})$ and $V$ is a $w^\ast$-neighborhood of $\varphi$. Then $\varPsi(\varphi)\in\widehat{\mathcal A}^{(\xi)}$ and $\varPsi(V\cap PS(\mathcal A))$ is neighborhood of $\varPsi(\varphi)$, so there is a $[H,\pi]\in\widehat{\mathcal A}^{(\zeta)}\cap\big(\varPsi(V\cap PS(\mathcal A))\setminus\{\varPsi(\varphi)\}\big)$. Therefore, there has a $\psi\in\varPsi^{-1}(\widehat{\mathcal A}^{(\zeta)})\cap V\subset s_\varepsilon^\zeta(B_{\mathcal A^\ast})\cap V$ such that $\varPsi(\psi)=[H_\psi,\pi_\psi]=[H,\pi]$. Since $[H_\psi,\pi_\psi]\neq[H_\varphi,\pi_\varphi]$, we can obtain
\[
\Vert\varphi-\psi\Vert\geq\big\vert\langle\varphi-\psi,z(\varphi)-z(\psi)\rangle\big\vert=2
\]
Therefore, ${\rm diam}\big(s_\varepsilon^\zeta(B_{\mathcal A^\ast})\cap V\big)\geq 2>\varepsilon$. It follows that $\varphi\in s_\varepsilon^{\zeta+1}(B_{\mathcal A^\ast})=s_\varepsilon^\xi(B_{\mathcal A^\ast})$. So
\[
\varPsi^{-1}\big(\widehat{\mathcal A}^{(\xi)}\big)\subset s_\varepsilon^\xi(B_{\mathcal A^\ast}).
\]
This finishes our induction.
\item [(ii)]
We will be done by transfinite induction on $\xi$. It is trivially true for $\xi=0$. It also passes easily to limit ordinal. So we let $\xi=\zeta+1$ and assume now that it is true for $\zeta$. For each $n\in\mathbb N$, let
\[
A_n=\bigg\{\frac{1}{2^n}{\sum}_{i=1}^{2^n}\varphi_i:\big\{[H_{\varphi_i},\pi_{\varphi_i}]\big\}_{i=1}^{2^n}\subset \widehat{\mathcal A}^{(\zeta)}~{\rm and}~[H_{\varphi_i},\pi_{\varphi_i}]\neq[H_{\varphi_j},\pi_{\varphi_j}]~{\rm for~any}~i\neq j\bigg\},
\]
\begin{align*}
B_n=\bigg\{&\frac{1}{2^n}{\sum}_{i=1}^{2^n}(\pi_\varphi(\cdot)e_i|e_i):[H_\varphi,\pi_\varphi]\in \widehat{\mathcal A}^{(\zeta)}~{\rm and}~\{e_i\}_{i=1}^{2^n}\\
&{\rm is~a~normalized~orthonormal~sequence~in}~H_\varphi\bigg\}.\\
\end{align*}
We claim that for each $n\in\mathbb N$, $A_n\cup B_n\subset d_\varepsilon^{\omega\zeta+n}(B_{\mathcal A^\ast})$. We show by induction on $n$. It is trivially true for $n=0$ by the inductive hypothesis. Assume that $n\geq1$ and the claim holds for $n-1$. Fix some $\phi=\frac{1}{2^n}{\sum}_{i=1}^{2^n}\varphi_i\in A_n$, $\psi=\frac{1}{2^n}{\sum}_{i=1}^{2^n}(\pi_\varphi(\cdot)e_i|e_i)\in B_n$. Then $\phi_1=\frac{1}{2^{n-1}}{\sum}_{i=1}^{2^{n-1}}\varphi_i,~\phi_2=\frac{1}{2^{n-1}}{\sum}_{i=2^{n-1}+1}^{2^n}\varphi_i\in d_\varepsilon^{\omega\zeta+n-1}(B_{\mathcal A^\ast})$ and $\psi_1=\frac{1}{2^{n-1}}{\sum}_{i=1}^{2^{n-1}}(\pi_\varphi(\cdot)e_i|e_i),~\psi_2=\frac{1}{2^{n-1}}{\sum}_{i=2^{n-1}+1}^{2^n}(\pi_\varphi(\cdot)e_i|e_i)\in d_\varepsilon^{\omega\zeta+n-1}(B_{\mathcal A^\ast})$. Moreover, $\phi=\frac{1}{2}\phi_1+\frac{1}{2}\phi_2$ and $\psi=\frac{1}{2}\psi_1+\frac{1}{2}\psi_2$. Since $d_\varepsilon^{\omega\zeta+n-1}(B_{\mathcal A^\ast})$ is convex, then $\phi,\psi\in d_\varepsilon^{\omega\zeta+n-1}(B_{\mathcal A^\ast})$. Let $S$ be a $w^\ast$-open slice containing $\phi$ and $T$ be a $w^\ast$-open slice containing $\psi$. The convexity yields that $S$ must contain either $\phi_1$ or $\phi_2$ and $T$ must contain either $\psi_1$ or $\psi_2$. Since $[H_{\varphi_i},\pi_{\varphi_i}]\neq[H_{\varphi_j},\pi_{\varphi_j}]$ for any $i\neq j$, it follows that $z(\varphi_i)\perp z(\varphi_j)$ for any $i\neq j$. We get
\begin{align*}
&\Vert\phi-\phi_1\Vert=\Vert\phi-\phi_2\Vert=\frac{1}{2}\Vert\phi_1-\phi_2\Vert\\
\geq&\frac{1}{2^n}\left\vert\left\langle{\sum}_{i=1}^{2^{n-1}}\varphi_i-{\sum}_{i=2^{n-1}+1}^{2^n}\varphi_i,{\sum}_{i=1}^{2^n-1}z(\varphi_i)-{\sum}_{i=2^{n-1}+1}^{2^n}z(\varphi_i)\right\rangle\right\vert=1>\varepsilon.\\
\end{align*}
On the other hand, we know that $\tilde{\pi_\varphi}(\mathcal A^{\ast\ast})=\mathcal B(H_\varphi)$. Therefore,
\begin{align*}
&\Vert\psi-\psi_1\Vert=\Vert\psi-\psi_2\Vert=\frac{1}{2}\Vert\psi_1-\psi_2\Vert\\
=&\frac{1}{2^n}\left\Vert{\sum}_{i=1}^{2^{n-1}}(\pi_\varphi(\cdot)e_i|e_i)-{\sum}_{i=2^{n-1}+1}^{2^n}(\pi_\varphi(\cdot)e_i|e_i)\right\Vert=1>\varepsilon.\\
\end{align*}
It follows that ${\rm diam}\big(S\cap d_\varepsilon^{\omega\zeta+n-1}(B_{\mathcal A^\ast})\big)>\varepsilon$ and ${\rm diam}\big(T\cap d_\varepsilon^{\omega\zeta+n-1}(B_{\mathcal A^\ast})\big)>\varepsilon$. This yields $A_n\cup B_n\subset d_\varepsilon^{\omega\zeta+n}(B_{\mathcal A^\ast})$. This finishes induction on $n$.

Now, for any $\varphi\in\varPsi^{-1}\big(\widehat{\mathcal A}^{(\xi)}\big)$ and $n\in\mathbb N$. For any convex $w^\ast$-neighborhood $V$ of $\varphi$. Then $\varPsi(V\cap PS(\mathcal A))$ is a neighborhood of $\varPsi(\varphi)$. Next, we will consider the following two cases. 

The first case, there is no $[H_\psi,\pi_\psi]\in\widehat{\mathcal A}^{(\zeta)}\setminus\big\{[H_\varphi,\pi_\varphi]\big\}$ such that $\ker[H_\psi,\pi_\psi]\subset\ker[H_\varphi,\pi_\varphi]$. We will inductively select $\{\varphi_i\}_{i=1}^{2^n}$ in $V\cap \varPsi^{-1}\big(\widehat{\mathcal A}^{(\zeta)}\big)$ such that $[H_{\varphi_i},\pi_{\varphi_i}]\neq[H_\varphi,\pi_\varphi]$ and $[H_{\varphi_i},\pi_{\varphi_i}]\neq[H_{\varphi_j},\pi_{\varphi_j}]$ for any $i\neq j$. Since $\varPsi(\varphi)\in\widehat{\mathcal A}^{(\xi)}$, we can pick a $\varphi_1\in V\cap PS(\mathcal A)$ with $[H_{\varphi_1},\pi_{\varphi_1}]\in\widehat{\mathcal A}^{(\zeta)}\setminus\{[H_\varphi,\pi_\varphi]\}$. If $\varphi_1,\cdots,\varphi_{k-1}$, $2\leq k\leq2^n$, have been chosen. Note that $\ker[H_{\varphi_i},\pi_{\varphi_i}]\nsubseteq\ker[H_\varphi,\pi_\varphi]$ for $i=1,\cdots,k-1$, it follows that $\varPsi(V\cap PS(\mathcal A))\setminus\bigcup_{i=1}^{k-1}{\rm hull}'\big(\ker[H_{\varphi_i}\pi_{\varphi_i}]\big)$ is a neighborhood of $\varPsi(\varphi)$. Therefore, we can choose a $\varphi_k\in V\cap PS(\mathcal A)$ such that $[H_{\varphi_k},\pi_{\varphi_k}]\in\mathcal A^{(\zeta)}\setminus\big[\{[H_\varphi,\pi_\varphi]\}\cup\bigcup_{i=1}^{k-1}{\rm hull}'\big(\ker[H_{\varphi_i}\pi_{\varphi_i}]\big)\big]$. It is easy to see that $[H_{\varphi_i},\pi_{\varphi_i}]\neq[H_\varphi,\pi_\varphi]$ for each $i$, $1\leq i\leq k$, and $[H_{\varphi_i},\pi_{\varphi_i}]\neq[H_{\varphi_j},\pi_{\varphi_j}]$ for any $1\leq i\neq j\leq k$. This completes the induction. Note that $\frac{1}{2^n}{\sum}_{i=1}^{2^n}\varphi_i\in V\cap A_n$ and $V$ is arbitrary, which implies that $\varphi\in\overline{A_n}^{w^\ast}\subset d_\varepsilon^{\omega\zeta+n}(B_{\mathcal A^\ast})$.

The second case, there is some $[H_\psi,\pi_\psi]\in\widehat{\mathcal A}^{(\zeta)}\setminus\big\{[H_\varphi,\pi_\varphi]\big\}$ such that $\ker[H_\psi,\pi_\psi]\subset\ker[H_\varphi,\pi_\varphi]$. Since $\mathcal A$ is tpye I, it follows that $\ker[H_\psi,\pi_\psi]\subsetneqq\ker[H_\varphi,\pi_\varphi]$ by \cite[Th 5.6.3]{Murphy}. Next, we will inductively select a normalized orthonormal sequence $\{e_i\}_{i=1}^{2^n}$ in $H_\psi$ such that $(\pi_\psi(\cdot)e_i|e_i)\in V$ for each $1\leq i\leq2^n$. To start the induction, we can pick $e_1\in H_\psi$ with $\Vert e_1\Vert=1$ such that $(\pi_\psi(\cdot)e_1|e_1)\in V$ by Lemma \ref{4.12}. If $e_1,\cdots,e_{k-1}$, $2\leq k\leq2^n$, have been chosen. By Lemma \ref{4.12}, there exists a $e_k\in H_\psi$ with $\Vert e_k\Vert=1$ and $e_k\perp{\rm span}\big(\{e_1,\cdots,e_{k-1}\}\big)$ such that $(\pi_\psi(\cdot)e_k|e_k)\in V$. This completes the induction. Note that $\frac{1}{2^n}{\sum}_{i=1}^{2^n}(\pi_\psi(\cdot)e_i|e_i)\in V\cap B_n$ and $V$ is arbitrary, which implies that $\varphi\in\overline{B_n}^{w^\ast}\subset d_\varepsilon^{\omega\zeta+n}(B_{\mathcal A^\ast})$.

Hence in both cases we get $\varphi\in d_\varepsilon^{\omega\zeta+n}(B_{\mathcal A^\ast})$ for all $n\in\mathbb N$. This implies that
\[
\varphi\in{\bigcap}_{n<\omega}d_\varepsilon^{\omega\zeta+n}(B_{\mathcal A^\ast})=d_\varepsilon^{\omega\zeta+\omega}(B_{\mathcal A^\ast})=d_\varepsilon^{\omega\xi}(B_{\mathcal A^\ast}).
\]
So $\varPsi^{-1}\big(\widehat{\mathcal A}^{(\xi)}\big)\subset d_\varepsilon^{\omega\xi}(B_{\mathcal A^\ast})$. This finishes our induction.
\end{itemize}
\end{proof}

\begin{remark}
Let $\mathcal A$ be a C*-algebra. If there are $\varphi,\psi\in PS(\mathcal A)$ such that $\ker(\pi_\varphi)=\ker(\pi_\psi)$, but $\{H_\varphi,\pi_\varphi\}$ is not unitarily equivalent to $\{H_\psi,\pi_\psi\}$. If $U$ is an open set of $\widehat{\mathcal A}$, it is easy to see that $[H_\varphi,\pi_\varphi]\in U$ if and only if $[H_\varphi,\pi_\varphi]\in U$. For any $\varepsilon\in(0,2)$ and ordinal $\xi$, it is easy to show that $\varPsi^{-1}\big(\{[H_\varphi,\pi_\varphi],[H_\psi,\pi_\psi]\}\big)\subset s_\varepsilon^\xi(B_{\mathcal A^\ast})$.
\end{remark}

For a real or complex linear space $X$ and $S$ is a non-empty subset of $X$. We denote by
\[
{\rm aco}(S)=\left\{{\sum}_{i=1}^k\lambda_ix_i:\lambda_1,\cdots,\lambda_k~{\rm are~scalars~with~}{\sum}_{i=1}^k\vert\lambda_i\vert\leq1,x_i\in S\right\}
\]
the $absolute~convex~hull$ of $S$.
\begin{lemma}\label{4.15}
Let $\mathcal A$ be a type I C*-algebra. If an ordinal $\xi<i\big(\widehat{\mathcal A}\big)$.
\begin{itemize}
\item [(i)] For any $\varepsilon\in(0,2)$, then $\overline{{\rm aco}}^\ast\big(\varPsi^{-1}\big(\widehat{\mathcal A}^{(\xi)}\big)\big)\subset s_\varepsilon^\xi(B_{\mathcal A^\ast})$.
\item [(ii)] For any $\varepsilon\in(0,1)$, then $\overline{{\rm aco}}^\ast\big(\varPsi^{-1}\big(\widehat{\mathcal A}^{(\xi)}\big)\big)\subset d_\varepsilon^{\omega\xi}(B_{\mathcal A^\ast})$.
\end{itemize}
\end{lemma}
\begin{proof}
(i): We will be done by transfinite induction on $\xi$. It is trivially true for $\xi=0$. It also passes easily to limit ordinal. So we let $\xi=\zeta+1$ and assume now that it is true for $\zeta$.

For any $\varphi\in{\rm aco}\big(\varPsi^{-1}\big(\widehat{\mathcal A}^{(\xi)}\big)\big)$. Then there is a finite non-empty subset $\big\{[H_1,\pi_1],\cdots,[H_m,\pi_m]\big\}$ of $\widehat{\mathcal A}^{(\xi)}$ such that for each $1\leq i\leq m$ there are $\{\varphi_{i,j}\}_{1\leq j\leq n_i}\subset PS(\mathcal A)$  and scalars $\{\lambda_{i,j}\}_{1\leq j\leq n_i}$ such that $[H_{\varphi_{i,j}},\pi_{\varphi_{i,j}}]=[H_i,\pi_i]$ for each $1\leq j\leq n_i$, $\sum_{i=1}^m\sum_{j=1}^{n_i}\vert\lambda_{i,j}\vert\leq1$ and
\[
\varphi={\sum}_{i=1}^m{\sum}_{j=1}^{n_i}\lambda_{i,j}\varphi_{i,j}.
\]
Let $V$ be any $w^\ast$-neighborhood of $\varphi$. Fix $x_1,\cdots,x_t\in\mathcal A$, and $\delta>0$ such that
\[
U=\{\phi\in\mathcal A^\ast:\vert\phi(x_i)-\varphi(x_i)\vert<\delta~{\rm and}~i=1,\cdots,t\}\subset V.
\]
Let $F$ be the those $i\in\{1,\cdots,m\}$ such that there is some $[H,\pi]\in\widehat{\mathcal A}^{(\zeta)}$ such that $\ker[H,\pi]\subsetneqq\ker[H_i,\pi_i]\}$. We can find a partition $\{F_1,\cdots,F_s\}$ of $F$ and $\big\{[H'_1,\pi'_1],\cdots,[H'_s,\pi'_s]\big\}\subset\widehat{\mathcal A}^{(\zeta)}$ such that for each $1\leq k\leq s$, $\ker[H'_k,\pi'_k]\subsetneqq\ker[H_i,\pi_i]$ for every $i\in F_k$. Using Lemma \ref{4.12} and by passing induction, for any $k=1,\cdots,s$, we can select a normalized orthonormal sequence $\{e_{i,j}:i\in F_k, 1\leq j\leq 2n_i\}$ in $H'_k$ such that $(\pi'_k(\cdot)e_{i,j}|e_{i,j})\in\varphi_{i,j}-\varphi+U$ for any $i\in F_k$ and $1\leq j\leq 2n_i$. Put
\[
\phi_{i,j}=(\pi'_k(\cdot)e_{i,j}|e_{i,j}),~~1\leq k\leq s,~i\in F_k,~{\rm and}~1\leq j\leq 2n_i.
\]
On that other hand, put $E=\{1,\cdots,m\}\setminus F$. For any $i\in E$, $1\leq j\leq n_i$, and $\big\{[K_1,\rho_1],\cdots,[K_n,\rho_n]\big\}\subset \mathcal A^{(\zeta)}$. Note that $\varPsi\big((\varphi_{i,j}-\varphi+U)\cap PS(\mathcal A)\big)\setminus\bigcup_{j=1}^n{\rm hull}'\big(\ker[K_j,\rho_j]\big)$ is a neighborhood of $[H_i,\pi_i]$, then we can choose a $\phi\in(\varphi_{i,j}-\varphi+U)\cap PS(\mathcal A)$ such that $[H_\phi,\pi_\phi]\in\mathcal A^{(\zeta)}\setminus\{[K_1,\rho_1],\cdots,[K_n,\rho_n]\}$. Therefore using this property and by passing induction, we can find a subset $\{\phi_{i,j}:i\in E,1\leq j\leq 2n_i\}$ such that $\phi_{i,j}\in\varphi_{i,j}-\varphi+U$ and
\[
[H_{\phi_{i,j}},\pi_{\phi_{i,j}}]\in\mathcal A^{(\zeta)}\setminus\big\{[H_{\phi_{u,v}},\pi_{\phi_{u,v}}]:u\neq i~{\rm or}~v\neq j\big\}\big].
\]
Put
\[
\phi={\sum}_{i=1}^m{\sum}_{j=1}^{n_i}\lambda_{i,j}\phi_{i,j}
\]
and
\[
\phi'={\sum}_{i=1}^m{\sum}_{j=1}^{n_i}\lambda_{i,j}\phi_{i,n_i+j}.
\]
If $F=\emptyset$. Then we can choose $\{\psi_1,\cdots,\psi_{2t+2}\}\subset\varphi_{1,1}-\varphi+U$ such that
\[
[H_{\psi_k},\pi_{\psi_k}]\in\mathcal A^{(\zeta)}\setminus\big\{[H_{\phi_{i,j}},\pi_{\phi_{i,j}}]:1\leq i\leq m~{\rm and}~1\leq j\leq 2n_i\big\}
\]
for any $k=1,\cdots,2t+2$, and
\[
[H_{\psi_k},\pi_{\psi_k}]\neq[H_{\psi_{k'}},\pi_{\psi_{k'}}],~1\leq k\neq k'\leq 2t+2.
\]
Since $\psi_k\in\mathcal A^\ast z(\psi_k)$, it is easy to see that $\{\psi_k\}_{k=1}^{2t+2}$ are linearly independent. Therefore, there exists
\[
\psi={\sum}_{k=1}^{t+1}a_k\psi_k,~\psi'={\sum}_{k=1}^{t+1}a'_k\psi_{t+1+k}\in{\bigcap}_{1\leq i\leq t}\ker\big(\langle x_i,\cdot\rangle\big)
\]
such that
\[
\Vert\phi+\psi\Vert=\left\Vert{\sum}_{i=1}^m{\sum}_{j=1}^{n_i}\lambda_{i,j}\phi_{i,j}z(\phi_{i,j})+{\sum}_{k=1}^{t+1}a_k\psi_kz(\psi_k)\right\Vert={\sum}_{i=1}^m{\sum}_{j=1}^{n_i}\vert\lambda_{i,j}\vert+{\sum}_{k=1}^{t+1}\vert a_k\vert=1
\]
and
\[
\Vert\phi'+\psi'\Vert=\left\Vert{\sum}_{i=1}^m{\sum}_{j=1}^{n_i}\lambda_{i,j}\phi_{i,n_i+j}z(\phi_{i,n_i+j})+{\sum}_{k=1}^{t+1}a'_k\psi_{t+1+k}z(\psi_{t+1+k})\right\Vert={\sum}_{i=1}^m{\sum}_{j=1}^{n_i}\vert\lambda_{i,j}\vert+{\sum}_{k=1}^{t+1}\vert a'_k\vert=1.
\]
By the induction hypothesis, it follows that $\phi+\psi,\phi'+\psi'\in{\rm aco}\big(\varPsi^{-1}\big(\widehat{\mathcal A}^{(\zeta)}\big)\big)\cap U\subset s_\varepsilon^\zeta(B_{\mathcal A^\ast})\cap V$. It is easy to show that
\[
\Vert\phi+\psi-(\phi'+\psi')\Vert=2>\varepsilon.
\]
If $F\neq\emptyset$. We choose some $i'\in F_1$. By Lemma \ref{4.12}, we know that $\dim(H_1)=\infty$. Then we can select a normalized orthonormal sequence $\{g_1,\cdots,g_{2t+2}\}$ in $H'_1$ which are orthogonal to $\{e_{i,j}:i\in F_1,1\leq j\leq 2n_i\}$. It is easy to see that $\big\{(\pi'_1(\cdot)g_1|g_1),\cdots,(\pi'_1(\cdot)g_{2t+2}|g_{2t+2})\big\}$ are linearly independent. Therefore, there exists some
\[
\sigma={\sum}_{k=1}^{t+1}b_k(\pi'_1(\cdot)g_k|g_k),~\sigma'={\sum}_{k=1}^{t+1}b'_k(\pi'_1(\cdot)g_{t+1+k}|g_{t+1k})\in{\bigcap}_{1\leq i\leq t}\ker\big(\langle x_i,\cdot\rangle\big)
\]
such that
\[
\Vert\phi+\sigma\Vert={\sum}_{i=1}^m{\sum}_{j=1}^{n_i}\vert\lambda_{i,j}\vert+{\sum}_{j=1}^{t+1}\vert b_j\vert=1.
\]
and
\[
\Vert\phi'+\sigma'\Vert={\sum}_{i=1}^m{\sum}_{j=1}^{n_i}\vert\lambda_{i,j}\vert+{\sum}_{j=1}^{t+1}\vert b'_j\vert=1.
\]
It follows that $\phi+\sigma,\phi+\sigma'\in{\rm aco}\big(\varPsi^{-1}\big(\widehat{\mathcal A}^{(\zeta)}\big)\big)\cap W\subset s_\varepsilon^\zeta(B_{\mathcal A^\ast})\cap V$. It is easy to show that
\[
\Vert\phi+\sigma-(\phi'+\sigma')\Vert=2>\varepsilon.
\]
To sum up, ${\rm diam}(V\cap s_\varepsilon^\zeta(B_{\mathcal A^\ast}))\geq2>\varepsilon$, this implies that $\varphi\in s_\varepsilon^\xi(B_{\mathcal A^\ast})$. So ${\rm aco}\big(\varPsi^{-1}\big(\widehat{\mathcal A}^{(\xi)}\big)\big)\subset s_\varepsilon^\xi(B_{\mathcal A^\ast})$. It follows that $\overline{{\rm aco}}^\ast\big(\varPsi^{-1}\big(\widehat{\mathcal A}^{(\xi)}\big)\big)\subset s_\varepsilon^\xi(B_{\mathcal A^\ast})$. This finishes our induction.

(ii): For any $a\in\mathbb C$ with $\vert a\vert=1$. We know that the map
\[
\varphi\in\mathcal A^\ast\mapsto a\varphi\in\mathcal A^\ast
\]
is $w^\ast$-$w^\ast$ homeomorphism and linearly isometric onto. It is easy to see that
\[
d_\varepsilon^{\omega\xi}(B_{\mathcal A^\ast})=d_\varepsilon^{\omega\xi}(a\cdot B_{\mathcal A^\ast})=a\cdot d_\varepsilon^{\omega\xi}(B_{\mathcal A^\ast}).
\]
By Lemma \ref{4.13} (ii), it follows that $d_\varepsilon^{\omega\xi}(B_{\mathcal A^\ast})\neq\emptyset$.
Choose a element $\varphi\in d_\varepsilon^{\omega\xi}(B_{\mathcal A^\ast})$, then $-\varphi\in d_\varepsilon^{\omega\xi}(B_{\mathcal A^\ast})$. By convexity of $d_\varepsilon^{\omega\xi}(B_{\mathcal A^\ast})$, which yields that $0=(\varphi-\varphi)/2\in d_\varepsilon^{\omega\xi}(B_{\mathcal A^\ast})$. We let $S_{\mathbb C}=\{a\in\mathbb C:\vert a\vert=2\}$. By Lemma \ref{4.13} (ii) and the convexity of $d_\varepsilon^{\omega\xi}(B_{\mathcal A^\ast})$, we can obtain that
\[
\overline{{\rm aco}}^\ast\big(\varPsi^{-1}\big(\widehat{\mathcal A}^{(\xi)}\big)\big)=\overline{{\rm co}}^\ast\left(\{0\}\cup \big[S_{\mathbb C}\cdot\varPsi^{-1}\big(\widehat{\mathcal A}^{(\xi)}\big)\big]\right)\subset d_\varepsilon^{\omega\xi}(B_{\mathcal A^\ast}).
\]
\end{proof}

Now, we are ready to state and prove the main theorem of this section as follows.
\begin{theorem}\label{4.16}
Let $\mathcal A$ be an infinite dimensional C*-algebra and $1<p<\infty$. Then
\[
Sz(\mathcal A)=\Gamma'(i(\mathcal A))~and~Dz(\mathcal A)=Sz(L_p(\mathcal A))=\omega Sz(\mathcal A)=\omega\Gamma'(i(\mathcal A)).
\]
\end{theorem}
\begin{proof}
If $\mathcal A$ is not scattered, then $i(\mathcal A)=\infty$ by Theorem \ref{4.5}. Conclusion is trivial.

If $\mathcal A$ is scattered, then $\mathcal A$ is type I by Proposition \ref{2.5}. Since ${\rm dim}(\mathcal A)=\infty$, then we can put $\Gamma'(i(\mathcal A))=\omega^{\alpha+1}$ for some ordinal $\alpha\geq0$. Then $\omega^\alpha\leq i(\mathcal A)<\omega^{\alpha+1}$. By Theorem \ref{3.2} and Lemma \ref{4.11}, we can obtain that
\[
Sz(\mathcal A)\leq\omega^{\alpha+1},~Dz(\mathcal A)\leq Sz(L_p(\mathcal A))\leq\omega^{1+\alpha+1}.
\]
 
On the other hand, for any $0<\varepsilon<1$, by Corollary \ref{4.7} and Lemma \ref{4.13}, it follows that
\[
Sz_\varepsilon(B_{\mathcal A^\ast})\geq\omega^\alpha,~Dz(\mathcal A)\geq\omega^{1+\alpha}.
\]
Since both $Sz_\varepsilon(B_{\mathcal A^\ast})$ and $Dz_\varepsilon(B_{\mathcal A^\ast})$ are successor ordinal and by Proposition \ref{3.1} (c), then
\[
Sz(\mathcal A)\geq\omega^{\alpha+1},~~Dz(\mathcal A)\geq\omega^{1+\alpha+1}.
\]
To sum up, we obtain
\[
Sz(\mathcal A)=\Gamma'(i(\mathcal A))~and~Dz(\mathcal A)=Sz(L_p(\mathcal A))=\omega Sz(\mathcal A)=\omega\Gamma'(i(\mathcal A)).
\]
\end{proof}

Combining Corollary \ref{4.7} and Theorem \ref{4.16} gives us the following:
\begin{corollary}\label{4.17}
Let $\mathcal A$ be an infinite dimension separable or type I C*-algebra. Then
\[
Sz(\mathcal A)=\Gamma'\big(i\big(\widehat{\mathcal A}\big)\big)~and~Dz(\mathcal A)=\omega\Gamma'\big(i\big(\widehat{\mathcal A}\big)\big)
\]
In particular, $\mathcal A$ is infinite dimension scattered C*-algebra, the same conclusion holds.
\end{corollary}

R.M. Causey \cite{Causey2017}  proved that the Szlenk index of an injective tensor product $X\check{\otimes}Y$ of Banach spaces $X$ and $Y$ is $\max\{Sz(X),Sz(Y)\}$, and \cite{Causey2022} proved that the Szlenk index of a projective tensor product $C(K)\widehat{\otimes}C
(L)$ of continuous functions on arbitrary scattered, compact, Hausdorff space is $\max\{Sz(C(K)),Sz(C(L))\}$. Next, we will show the similar conclusion holds for the C*-tensor product of C*-algebras.

\begin{theorem}\label{4.18}
Let $\mathcal A$ and $\mathcal B$ be non-zero C*-algebras, $\mathcal A\otimes_\beta\mathcal B$ is a C*-tensor product of $\mathcal A$ and $\mathcal B$. Then
\[
Sz(\mathcal A\otimes_\beta\mathcal B)=\max\{Sz(\mathcal A),Sz(\mathcal B)\},
\]
\[
Dz(\mathcal A\otimes_\beta\mathcal B)=\max\{Dz(\mathcal A),Dz(\mathcal B)\}.
\]
\end{theorem}
In order to proof the Theorem \ref{4.18}, we need the following lemma:

\begin{lemma}\label{4.19}
Let $X$ and $Y$ be topological spaces. $X\times Y$ is the product topological space of $X$ and $Y$. Then for any ordinal $\alpha$ and $n\in\mathbb N$,
\[
(X\times Y)^{(\omega^\alpha\cdot n)}=\bigcup_{r+s=n\atop r,s\geq0}X^{(\omega^\alpha\cdot r)}\times Y^{(\omega^\alpha\cdot s)}.
\]
\end{lemma}

A proof of the Lemma \ref{4.19} can be found in \cite{Sahan}, for the sake of readability, we give here a sketch proof. 

\begin{proof}
We will be done by transfinite induction on $\alpha$. It is easy to see that
\[
(X\times Y)'=X'\times Y\bigcup X\times Y',
\]
so it is true for $\alpha=0$ and $n=1$. Now we assume for all ordinal $\gamma<\alpha$, conclusion is true. 

If $\alpha$ is a limit ordinal. It is easy to see that
\[
(X\times Y)^{(\omega^\alpha)}=\bigcap_{\gamma<\alpha}(X\times Y)^{(\omega^\gamma)}=\bigcap_{{\gamma<\alpha}}X^{(\omega^\gamma)}\times Y\cup X\times Y^{(\omega^\gamma)}=X^{(\omega^\alpha)}\times Y\bigcup X\times Y^{(\omega^\alpha)}.
\]
Now, we will inductively prove that $(X\times Y)^{(\omega^\alpha\cdot n)}=\bigcup_{r+s=n\atop r,s\geq0}X^{(\omega^\alpha\cdot r)}\times Y^{(\omega^\alpha\cdot s)}$ on $n$. When $n=0$, it is trivially true. Now we let $n=k+1$ and assume that is true for $k$. Then $(X\times Y)^{(\omega^\alpha\cdot k)}=\bigcup_{r+s=k\atop r,s\geq0}X^{(\omega^\alpha\cdot r)}\times Y^{(\omega^\alpha\cdot s)}$. By the induction hypothesis, 
\[
(X\times Y)^{(\omega^\alpha\cdot(k+1))}=\left((X\times Y)^{(\omega^\alpha\cdot k)}\right)^{(\omega^\alpha)}=\bigcup_{r+s=k\atop r,s\geq0}\left(X^{(\omega^\alpha\cdot r)}\times Y^{(\omega^\alpha\cdot s)}\right)^{(\omega^\alpha)}
\]
\[
=\bigcup_{r+s=k\atop r,s\geq0}X^{(\omega^\alpha\cdot(r+1))}\times Y^{(\omega^\alpha\cdot s)}\cup X^{(\omega^\alpha\cdot r)}\times Y^{(\omega^\alpha\cdot(s+1))}=\bigcup_{r+s=n\atop r,s\geq0}X^{(\omega^\alpha\cdot r)}\times Y^{(\omega^\alpha\cdot s)}.
\]
This finishes induction on $n$.

If $\alpha=\beta+1$. By the same method as before, that $(X\times Y)^{(\omega^\beta\cdot n)}=\bigcup_{r+s=n\atop r,s\geq0}X^{(\omega^\beta\cdot r)}\times Y^{(\omega^\alpha\cdot s)}$. It follows that
\[
(X\times Y)^{(\omega^\alpha)}=\bigcap_{n\geq0}\bigcup_{r+s=n\atop r,s\geq0}X^{(\omega^\beta\cdot r)}\times Y^{(\omega^\beta\cdot s)}=X^{(\omega^\alpha)}\times Y\bigcup X\times Y^{(\omega^\alpha)}.
\]
and
\[
(X\times Y)^{(\omega^\alpha\cdot n)}=\bigcup_{r+s=n\atop r,s\geq0}X^{(\omega^\alpha\cdot r)}\times Y^{(\omega^\alpha\cdot s)}.
\]
This finishes induction on $\alpha$.
\end{proof}

\begin{proof}[Proof of Theorem \ref{4.18}]
If $\dim(\mathcal A)=n<\infty$. Then $\mathcal A\otimes_\beta\mathcal B$ is linear isomorphic to $\overbrace{\mathcal B\oplus\cdots\oplus\mathcal B}^{\rm{n~times}}$. So
\[
Sz(\mathcal A\otimes_\beta\mathcal B)=Sz(\mathcal B)=\max\{Sz(\mathcal A),Sz(\mathcal B)\}.
\]
If $\mathcal A$ is not scattered. We choose a $b\in\mathcal B$ satisfying $\Vert b\Vert=1$. Then $\mathcal A$ is linearly isometric to the subspace $\{a\otimes b:a\in\mathcal A\}$ of $\mathcal A\otimes_\beta\mathcal B$, so
\[
Sz(\mathcal A\otimes_\beta\mathcal B)=\infty=\max\big\{Sz(\mathcal A),Sz(\mathcal B)\big\}.
\]

We can suppose without loss of generality that $\mathcal A$ and $\mathcal B$ are infinite dimension scattered C*-algebras. This implies $\mathcal A$ and $\mathcal B$ are type I, therefore
\[
\widehat{\mathcal A\otimes_\beta\mathcal B}=\widehat{\mathcal A\otimes_{\rm{max}}\mathcal B}\simeq\widehat{\mathcal A}\times\widehat{\mathcal B}
\]
by \cite{Guichardet}. We let $\max\left\{\Gamma'\big(i\big(\widehat{\mathcal A}\big)\big),\Gamma'\big(i\big(\widehat{\mathcal B}\big)\big)\right\}=\omega^{\alpha+1}$, then there is $n\in\mathbb N$ such that $\widehat{\mathcal A}^{(\omega^\alpha\cdot n)}=\emptyset$ and $\widehat{\mathcal B}^{(\omega^\alpha\cdot n)}=\emptyset$. By the Lemma \ref{4.19}, it is easy to see that $\left(\widehat{\mathcal A}\times\widehat{\mathcal B}\right)^{(\omega^\alpha\cdot2n)}=\emptyset$. So $i\big(\widehat{\mathcal A}\times\widehat{\mathcal B}\big)\leq\omega^\alpha\cdot2n<\omega^{\alpha+1}$, which implies $\Gamma'\left(i\big(\widehat{\mathcal A}\times\widehat{\mathcal B}\big)\right)=\omega^{\alpha+1}$. Therefore,
\[
Sz(\mathcal A\otimes_\beta\mathcal B)=\max\{Sz(\mathcal A),Sz(\mathcal B)\}
\]
by the Corollary \ref{4.17}. On the other hand,
\[
Dz(\mathcal A\otimes_\beta\mathcal B)=\omega Sz(\mathcal A\otimes_\beta\mathcal B)=\max\{\omega Sz(\mathcal A),\omega Sz(\mathcal B)\}=\max\{Dz(\mathcal A),Dz(\mathcal B)\}.
\]
\end{proof}
\begin{remark}
If $\mathcal A\otimes_\beta\mathcal B=\mathcal A\check{\otimes}\mathcal B$, where is $\mathcal A\check{\otimes}\mathcal B$ is the injective tensor product of $\mathcal A$ and $\mathcal B$, then either $\mathcal A$ of $\mathcal B$ is abelian (e.g. see book \cite[Page 211]{Takesaki}).
\end{remark}
\begin{definition}
	Let $\mathcal A$ be a C*-algebra and $a\in\mathcal A_h$ (the set of self-adjoint elements of $\mathcal A$). Denote by $C^\ast(a)$ the C*-subalgebra of $\mathcal A$ generated by $a$.
\end{definition}
Suppose that $\mathcal A$ be a C*-algebra, is there a commutative C*-subalgebra $\mathcal C$ of $\mathcal A$ such that
\[
Sz(\mathcal A)=Sz(\mathcal C)?~Dz(\mathcal A)=Dz(\mathcal C)?
\] 
When we consider situation for nonseparable C*-algebras. C. Hida and P. Koszmider \cite{HK} show that assuming an additional axiom (introduced by R. Jensen), there is a full noncommutative nonseparable scattered (Definition see \cite[Page 994]{HK}) C*-algebra $\mathcal B$ (of operator in $\mathcal B(\ell_2)$) with no nonseparable commutative subalgebra and with no uncountable irredundant set. By the \cite[Prop. 3.16]{HK}, then $\mathcal B$ contains a C*-subalgebra $\mathcal B_1$ such that $i(\mathcal B_1)=\omega_1=\omega^{\omega_1}$ ($\omega_1$ denote the first uncountable ordinal number). By Corollary \ref{4.17}, this implies $Sz(\mathcal B_1)=Dz(\mathcal B_1)=\omega_1\cdot\omega$. Let $\mathcal D$ is a commutative C*-subalgebra of $\mathcal B$, then $\mathcal D$ is separable and 
scattered, so 
\[
Sz(\mathcal D)\leq Dz(\mathcal D)<\omega_1<Sz(\mathcal B)=Dz(\mathcal B).
\]
The situation for separable C*-algebra is positive. H.X. Lin \cite[Th. 5.4]{Lin} studied the properties of self-adjoint elements on scattered C*-algebras. Inspired by the conclusion and proof method, we now adopt the techniques over there and Theorem \ref{4.16} to show the following theorem.
\begin{theorem}
	Let $\mathcal A$ be a separable C*-algebra. Then there exists $a\in\mathcal A_h$ such that
	\[
	Sz(\mathcal A)=Sz(C^\ast(a))~and~Dz(\mathcal A)=Dz(C^\ast(a))
	\] 
\end{theorem}
\begin{proof}
	If $\mathcal A$ is not scattered, then $i(\mathcal A)=\infty$. By Proposition \ref{2.5}, there exists $a\in\mathcal A_h$ such that $C^\ast(a)\simeq C_0((0,1])$, so 
	\[
	Sz(\mathcal A)=Sz(C^\ast(a))=\infty.
	\]
	
	If $\mathcal A$ is scattered. In order to conclude proof of this conclusion, we will show that there exists $a\in\mathcal A_h$ such that 
	\[
	i(\sigma(a))\geq i(\mathcal A).
	\]
	Since $\mathcal A$ is a separable scattered C*-algebra, then $\mathcal A^\ast$ is separable by Proposition \ref{2.5}  and Theorem \ref{2.1}. So $\widehat{\mathcal A}$ is countable by \cite[Th. 3.1]{Jensen1977}, this implies $i(\mathcal A)=i(\widehat{\mathcal A})<\omega_1$. We let $i(\mathcal A)=\alpha$. We will be done by transfinite induction on $\alpha$. 
	
	It is trivially true for $\alpha=1$. 
	
	If $\alpha>1$ is limit ordinal, and assume now that it is true for any ordinal less than $\alpha$. We choose a strictly increasing sequence $(\alpha_n)$ ordinals satisfying $\alpha_n\nearrow\alpha$ and a sequence of non-zero projection $\overline{p}_n\in\mathcal A^{(\alpha_{n+1})}/\mathcal A^{(\alpha_n)}$. Let $\phi_n:\mathcal A^{(\alpha_{n+1})}\to \mathcal A^{(\alpha_{n+1})}/\mathcal A^{(\alpha_n)}$ be the canonical homomorphism. By \cite[Lem. 5.1]{Lin} and the projection lifting theorem \cite{Brown}, there is a projection $p_1\in\mathcal A^{(\alpha_2)}$ such that $\phi_1(p_1)=\overline{p}_1$. Assume that mutually orthogonal projections $p_k\in\mathcal A^{(\alpha_{k+1})}$, $k=1,\dots,m$, have been chosen, such that $\phi_k(p_k)=\overline{p}_k$. We now let $q_m=p_1+\cdots+p_m$. Using the projection lifting theorem on $(1-q_m)\mathcal A^{(\alpha_{m+2})}(1-q_m)/\mathcal A^{(\alpha_{m+1})}\simeq\mathcal A^{(\alpha_{m+2})}/\mathcal A^{(\alpha_{m+1})}$, there is a projection $p_{m+1}\in (1-q_m)\mathcal A^{(\alpha_{m+2})}(1-q_m)$ such that $\phi_{m+1}(p_{m+1})=\overline{p}_{m+1}$ and $p_1,\dots,p_m,p_{m+1}$ are mutually orthogonal. By passing induction, we can obtain a mutually orthogonal projections sequence $(p_n)_{n=1}^\infty$ such that $\phi_n(p_n)=\overline{p}_n$. Since $\phi_n(p_n)\neq0$, there is $[H,\pi]\in{\rm hull}'(\mathcal A^{(\alpha_n)}))={\rm hull}'(\ker'(\widehat{\mathcal A}^{(\alpha_n)}))=\widehat{\mathcal A}^{(\alpha_n)}$ such that $\pi(p_n)\neq0$. Because $p_n\mathcal A p_n$ is a hereditary C*-subalgebra of $\mathcal A$ and $(p_n\mathcal A p_n)^\wedge$ is homeomorphic to $\widehat{\mathcal A}\setminus{\rm hull}'(p_n\mathcal A p_n)$ that is a open set of $\widehat{\mathcal A}$. It is easy to see that
	\[
	i(p_n\mathcal A p_n)=i((p_n\mathcal A p_n)^\wedge)\geq\alpha_n+1.
	\]
So $i((p_n\mathcal A p_n)^{(\alpha_n+1)})=\alpha_n+1$. By the induction hypothesis, there is $a_n\in{(p_n\mathcal A p_n)^{(\alpha_n+1)}}_h$, $\Vert a_n\Vert\leq 1$ such that
	\[
	i(\sigma(a_n))\geq i((p_n\mathcal A p_n)^{(\alpha_n+1)})=\alpha_n+1.
	\]
	Now, we let
	\[
	a=\sum_{n=1}^\infty\frac{1}{2^n}a_n.
	\]
	Then $\sigma(a)\supset\frac{1}{2^n}\sigma(a_n)$ for all $n\in\mathbb N$, it follows that $i(\sigma(a))\geq\alpha$.
	
	If $\alpha=\beta+1$, and assume now that it is true for $\beta$. If $\beta$ is a limit ordinal, by the induction hypothesis, there is $a\in\mathcal A^{(\beta)}$ such that $i(\sigma(a))\geq\beta$. Since $\sigma(a)$ is compact, this implies $i(\sigma(a))$ is a successor ordinal. So $i(\sigma(a))\geq\beta+1=\alpha$. If $\beta=\gamma+1$, it is easy to see that $\mathcal A/\mathcal A^{(\gamma)}$ is infinite dimension. So $\mathcal A/\mathcal A^{(\gamma)}$ contain a infinite dimensional commutative C*-subalgebra that it is scattered. It follows that $\mathcal A/\mathcal A^{(\gamma)}$ contains a sequence of mutually orthogonal non-zero projections $(\overline e_n)_{n=1}^\infty$. Let $\phi:\mathcal A\to\mathcal A/\mathcal A^{(\gamma)}$. Same as the previous proof, there is a mutually orthogonal projections sequence $(e_n)_{n=1}^\infty$ such that $\phi(e_n)=\overline e_n$, and $i((e_n\mathcal A e_n)^{(\beta)})=\beta$. By the induction hypothesis, there is $b_n\in{(e_n\mathcal A e_n)^{(\beta)}}_h$, $\Vert b_n\Vert<1$ such that
	\[
	i(\sigma(b_n))\geq i((e_n\mathcal A e_n)^{(\beta)})=\beta.
	\]
	Now, we let
	\[
	a=\sum_{b=1}^\infty\frac{1}{2^n}(e_n+b_n).
	\]
Since $\sigma(e_n+b_n)\subset(0,2]$, $\emptyset\neq\sigma(e_n+b_n)^{(\gamma)}\subset(0,2]$, and $\frac{1}{2^n}\sigma(e_n+b_n)\subset\sigma(a)$. It follows that $0\in\sigma(a)^{(\beta)}$, so
\[
i(\sigma(a))\geq\beta+1=\alpha.
\] 
This finishes our induction.

Now, we choose a $a\in\mathcal A_h$ such that $i(\sigma(a))\geq i(\mathcal A)$. Note that
\[
Sz(C^\ast(a))=Sz(C(\sigma(a))).
\]
So
\[
Sz(C^\ast(a))=\Gamma'(i(\sigma(a)))\geq\Gamma'(i(\mathcal A))=Sz(\mathcal A).
\]
Therefore, $Sz(\mathcal A)=Sz(C^\ast(a))$. On the other hand,
\[
Dz(\mathcal A)=\omega Sz(\mathcal A)=\omega Sz(C^\ast(a))=Dz(C^\ast(a)).
\]
\end{proof}

\section{\large }

\section*{\large Acknowledgment}
The authors would like to thank people in the Functional Analysis seminar of Xiamen University for their helpful conversations on the paper.

\bibliographystyle{amsalpha}

\end{document}